\documentclass[12pt,a4paper,final]{amsart}
\usepackage{amssymb}
\usepackage{amsthm}
\usepackage{amsmath}
\usepackage{mathrsfs}

\numberwithin{equation}{section}

\usepackage{geometry}
\newtheorem{theorem}{Theorem}[section]

\newtheorem{definition}[theorem]{Definition}
\newtheorem{proposition}[theorem]{Proposition}
\newtheorem{corollary}[theorem]{Corollary}
\newtheorem{lemma}[theorem]{Lemma}

\theoremstyle{remark}
\newtheorem{rmk}[theorem]{Remark}
\newtheorem{examp}[theorem]{Example}

\newcommand{\RR}{{{\mathbb R}^3}}
\newcommand{\R}{{{\mathbb R}}}
\newcommand{\C}{{{\mathbb C}}}
\newcommand{\N}{{{\mathbb N}}}

\newcommand{\PM}{{{\mathcal{PM}}}}
\renewcommand{\Re}{{{{\rm Re\,}}}}
\renewcommand{\Im}{{{{\rm Im\,}}}}
\newcommand{\vf}{{\varphi}}
\newcommand{\VF}{{\Phi}}
\newcommand{\wvf}{{\widetilde\varphi}}
\newcommand{\wpsi}{{\widetilde\psi}}

\newcommand{\dn}{{\,d\sigma}}
\newcommand{\wa}{{\widetilde{\alpha}}}

\newcommand{\Xa}{{{\mathcal{X}}^\alpha}}
\newcommand{\XaT}{{{\mathcal{X}}^\alpha_T}}
\newcommand{\F}{{{\mathcal F}}}
\newcommand{\G}{{{\mathcal G}}}
\newcommand{\B}{{{\mathcal B}}}

\newcommand{\PDa}{{{\mathcal{K}}^\alpha}}
\newcommand{\g}{{\gamma_2}}
\newcommand{\ga}{{\gamma_\alpha}}

\newcommand{\rf}[1]{(\ref{#1})}

\title[Homogeneous Boltzmann equation]{
Infinite energy solutions  to the homogeneous Boltzmann equation}

\author{Marco Cannone}
\address
{Universit\'e Paris-Est, Laboratoire d'Analyse
  et de Math\'ematiques 
Appliqu\'ees,  UMR 8050 CNRS,
  5 boulevard Descartes, Cit\'e  Descartes Champs-sur-Marne, 
 77454 Marne-la-Vall\'ee cedex 2, France}
\email{marco.cannone@univ-mlv.fr}

\author{Grzegorz Karch}
\address{
 Instytut Matematyczny, Uniwersytet Wroc\l awski,
 pl. Grunwaldzki 2/4, 50-384 Wroc\-\l aw, Poland}
\email{grzegorz.karch@math.uni.wroc.pl}
\urladdr{http://www.math.uni.wroc.pl/~karch}

\thanks{
This work was partially supported by
by the European Commission Marie Curie Host Fellowship for the Transfer of 
Knowledge
``Harmonic Analysis, Nonlinear Analysis and Probability''  MTKD-CT-2004-013389,
and by the Polish Ministry of Science grant N201 022 32/0902.}



\date{\bf \today}

\begin{document}

\begin{abstract}
The goal of  this work is to present an approach to
the  homogeneous Boltzmann equation  for  Maxwellian molecules
with a physical  collision kernel
which allows us  to construct unique solutions to the initial value problem
in a space of probability measures defined via the Fourier transform.
In that space,  the second moment of a measure is not assumed to be
finite, so infinite energy solutions are not {\it a priori} excluded
from our considerations.
Moreover, we study the large time asymptotics
of  solutions and, in a particular case, we  give an elementary
proof of the asymptotic stability of
self-similar solutions obtained by
A.V. Bobylev and  C.~Cer\-ci\-gna\-ni
[J. Stat. Phys. {\bf 106} (2002), 1039--1071].\\
\\
To appear in {\it Communications on Pure and Applied Mathematics.}
\end{abstract}

\maketitle



\tableofcontents

\baselineskip=17.5pt


\section{Introduction}

We consider
the homogeneous Boltzmann
equation 
in $\RR$ 
\begin{equation}
\partial_t f(v,t) = Q(f,f)(v,t) \label{eq} 
\end{equation}
with the bilinear form corresponding to a Maxwellian gas
\begin{equation}
Q(g,f)(v)=
\int_{\RR} \int_{S^2}
\B\left(\frac{v-v_*}{|v-v_*|}\cdot \sigma \right)
\big(f(v')g(v'_*)-f(v)g(v_*)\big)
\dn\,dv_*.
\label{B}
\end{equation}
Here,   the unknown density $f=f(v,t)$ is independent of the space
variable, moreover, we denote
\begin{equation}
v'=\frac{v+v_*}{2}+\frac{|v-v_*|}{2}\sigma, \quad
v'_*=\frac{v+v_*}{2}-\frac{|v-v_*|}{2}\sigma
\label{v}
\end{equation}
with $\sigma$ varying in the unit sphere $S^2$.
 Equation \rf{eq}--\rf{B} is supplemented with a nonnegative initial datum 
\begin{equation}\label{ini}
f(v,0)=f_0(v)
\end{equation}
 which is assumed  to be  a density of  a 
probability distribution (or, more generally, a probability measure).

The collision kernel $\B$ in \rf{B} is supposed to be a nonnegative function
and, in the case of Maxwellian molecules, it depends only on the deviation 
angle $\theta$, defined by the equation 
$\cos \theta = \frac{v-v_*}{|v-v_*|}\cdot \sigma$.
It is well-known that  the physical collision  kernel 
 $\B=\B(s)$  has a nonintegrable 
singularity as $s\to 1$ of the form $(1-s)^{-5/4}$ 
(see {\it e.g.}   
\cite[p.~1043]{BC02-self}, \cite[Ch.~1.1]{V} and  references therein). 
By the method developed 
in this work can, we can handle this kind of non-integrability 
as well as other singular kernels 
$\B$,
see Remark \ref{rem:kernel} for more details.

In the study of the Boltzmann equation, it is natural to assume that the 
nonnegative initial datum satisfies
\begin{equation}\label{as:f0}
\int_{\RR}f_0(v)\,dv=1, \quad \int_{\RR}f_0(v)v_i\,dv=0\; (i=1,2,3),\quad 
\int_{\RR}f_0(v)|v|^2\,dv=3,
\end{equation}
because 
these relations are interpreted as the unit mass, the zero mean value, and
the unit temperature of the gas, respectively. 
The existence of a unique solution
of the initial value problem \rf{eq}--\rf{ini} under  assumptions \rf{as:f0}
and for a large class on nonintegrable collision kernels is well-known,
see {\it e.g.} \cite{B2, PT96,V} and the references therein. 
This solution 
satisfies
$f\in C^1([0,\infty), L^1(\RR))$  and 
\begin{equation}\label{as:f1}
\int_{\RR}f(v,t)\,dv=1, \;\; \int_{\RR}f(v,t)v_i\,dv=0\; (i=1,2,3),\;\;
\int_{\RR}f(v,t)|v|^2\,dv=3
\end{equation}
for all $t>0$. For more information about the Boltzmann equation and its 
physical meaning, we refer the reader to the  book by Cercignani \cite{C}
and to the more recent review article by Villani \cite{V}.

In this work, we propose a method of studying properties of solutions to 
problem  \rf{eq}--\rf{ini} under very weak assumptions on the 
collision kernel  in which we do not need to assume that the second
moment of the unknown is finite. Hence, solutions with infinite temperature 
(or infinite energy) are not excluded {\it a priori} from our considerations.
These solutions are important because, as described by 
Bobylev and Cercignani  \cite{BC02-self}, they are connected to the shock-wave 
problem.

We limit ourselves to the study of the homogeneous  
Boltzmann equation for Max\-wel\-lian molecules. In fact, our reasoning will be 
based on the important observation by
 Bobylev \cite{B1,B2} showing that, in this case, 
the bilinear form \rf{B} can be easily studied by the Fourier transform. 
More precisely, denoting
\begin{equation}\label{Fourier}
\vf(\xi,t) \equiv \widehat f(\xi,t)= \int_{\RR} e^{-iv\cdot \xi} f(v,t)\;dv
\end{equation}
Bobylev was able to convert equation \rf{eq} into  the following equation 
for the new unknown $\vf=\vf(\xi,t)$
\begin{equation}
\partial_t\vf(\xi,t) = \int_{S^2} \B
\left(\frac{\xi\cdot \sigma}{|\xi|}\right)
\big(\vf(\xi^+,t)\vf(\xi^-,t)-\vf(\xi,t)\vf(0,t)\big)\dn
\label{Feq}
\end{equation}
where 
\begin{equation}\label{xipm}
\xi^+=\frac{\xi+|\xi|\sigma}{2}, \quad \xi^-=\frac{\xi-|\xi|\sigma}{2} 
\end{equation}
and we recall  that these two vectors $\xi^+$ and $\xi^-$ satisfy the 
well-known
relations
\begin{equation}\label{xipm1}
\xi^++\xi^-=\xi \quad\mbox{and} \quad |\xi^+|^2+|\xi^-|^2=|\xi|^2,
\end{equation}
hence, 
\begin{equation}\label{xipm2}
|\xi^+|^2=|\xi|^2\frac{1+\frac{\xi}{|\xi|}\cdot \sigma}2
\quad \mbox{and}\quad 
|\xi^-|^2=|\xi|^2\frac{1-\frac{\xi}{|\xi|}\cdot \sigma}2.
\end{equation}
We also note that the formula 
for the Fourier transform of the bilinear operator $Q$ 
on the right-hand side of \rf{Feq} is actually 
a particular case of a more general one which does not assume Maxwellian 
collision kernel, see \cite[Appendix]{ADVW} for more details.
 
In the following, we study properties of solutions to 
 equation \rf{Feq} supplemented with an initial datum
\begin{equation}
\vf(\xi, 0)=\vf_0(\xi). \label{Fini}
\end{equation}

All our results on solutions of \rf{eq}--\rf{ini} are formulated  
for the initial value problem 
in the Fourier variables \rf{Feq}--\rf{Fini} in the space of
characteristic functions (see Definition~\ref{def:CF}).


Motivated by a series of papers by  Toscani and coauthors \cite{GTW, CGT, TV}, 
we  study the problem  \rf{Feq}--\rf{Fini} in 
the function space described, in the Fourier variables and 
for suitable values of the parameter 
$\alpha$, by the following  pseudo-norm  
\begin{equation}
\|\vf\|_\alpha\equiv\sup_{\xi\in\RR} \frac{|\vf(\xi)|}{|\xi|^{\alpha}}.
 \label{pseudo}
\end{equation}
For $\alpha=0$ the quantity \rf{pseudo} defines the space ${\PM}$ of pseudo-measures, {\it i.e.}~tempered distributions, whose 
Fourier transforms are  bounded functions.
Moreover, we notice that for positive $\alpha$ the quantity 
$\|\vf\|_\alpha$ 
describes the behavior of $\vf$ at zero ({\it i.e.} the moments of the inverse 
Fourier transform of $\vf$) 
and for $\alpha$ negative $\|\vf\|_\alpha$ characterizes the behavior of 
$\vf$  at infinity ({\it i.e.}  the regularity
of the inverse Fourier transform of $\vf$). 

However, this quantity  is not a norm, in general. 
For example, when   $\alpha> 
0$, the number $\|\vf\|_\alpha$ is finite  if the inverse Fourier 
transform of $\vf$  
has polynomial moments of (high enough) degree  equal to zero. 
On the other 
hand, if $\alpha<0$, than $\|\vf\|_\alpha = 0$ for any tempered distribution 
$\vf$,  whose inverse Fourier transform is a polynomial of  
certain (not too high) degree (see {\it e.g.} \cite[Ch.~4]{FJW}).
It is easy to verify that if we work modulo suitable equivalence classes, than 
$\|\vf\|_\alpha$ is a norm 
and, as it was noticed in  \cite{CP}, this norm corresponds to the generalized 
homogeneous Besov space $\dot B^{-\alpha,\infty}_{\PM}$, 
based on the pseudo-measure space ${\PM}$ of tempered distributions.

For the Navier-Stokes and other
parabolic equations, it is well-known that
the regularity (in space) of the solution plays an important role, so that it is
natural to consider negative values of $\alpha$ in this context. Le Jan and Sznitman \cite{LJS} introduced the
scaling invariant norm $\|\cdot\|_\alpha$ with $\alpha= -2$ for the
Navier-Stokes equations. In \cite{CK04}, following this approach, we obtained the existence and the large time
asymptotic of
{\it infinite energy} solutions to the incompressible Navier- Stokes system in
the space $\dot B^{2,\infty}_{\PM}$.
A similar approach was introduced in \cite{BCKG} for the study of a
model of gravitating particles (see \cite{C04}, for a review).

At variance with the Navier-Stokes system, in the case of the homogeneous  
Boltzmann equation \rf{eq}, 
the space integrability of a solution plays a pivotal role. 
It means that we should take into account
the behavior of 
the Fourier transform  of a solution as $|\xi|\to 0$ and not when  $|\xi|\to 
\infty$. In other words, if it is natural to take negative values of $\alpha$ 
for the Navier-Stokes equations and other parabolic systems,  positive values of 
$\alpha$ should 
be considered for the Boltzmann equation.

In this direction, for $\alpha\geq 2 $, 
Toscani and coauthors \cite{GTW, CGT, TV} 
were able to obtain several nice results for the homogeneous Boltzmann equation  
(see Villani \cite{V}, for a review).  For example,  in the case of $\alpha=2$, 
Toscani and Villani \cite{TV}  proved 
the uniqueness and the stability 
of solutions to the homogeneous Boltzmann equation
for Maxwellian gas with a physical nonintegrable collision kernel.
Their proof required the energy of the solution 
({\it i.e.}~the last equality in \rf{as:f1}) to be finite.

In this work, we  treat the case  $0\leq \alpha\leq 2$
and we study the initial value problem \rf{Feq}--\rf{Fini}
 in a larger space 
where infinite energy
solutions are not excluded {\it a priori}. In this setting, 
our norms  growth exponentially in time 
and the 
trend to equilibrium  will be described in self-similar variables. We do not
impose the Grad cut-off assumption : 
any collision kernel satisfying 
$(1-s^2)^{\alpha_0/4}\B(s) \in L^1(-1,1)$ for some $\alpha_0\in [0,2]$ will be 
included in our 
approach  
(see Remark \ref{rem:kernel} for more details).


\section{Main results}
We begin by defining the function set which plays the main role in 
our study of properties of solution to the initial value problem
\rf{Feq}--\rf{Fini}.
First, recall that any solution $f=f(\cdot, t)$ 
of the homogeneous Boltzmann equation \rf{eq}--\rf{B} 
is (after the well-known normalization)
a probability 
measure for every $t\geq 0$.
Following the probabilistic terminology, we are going to use the set 
$\mathcal{K}$ of ``{\it characteristic functions}'', 
{\it i.e.} those functions that are  Fourier transforms
of probability measures ({\it cf.}~Definition \ref{def:CF}).
 In the next section, we 
will also introduce a  
more general set consisting of  ``{\it positive definite functions}" 
({\it cf.}~Definition \ref{def:pos:def}).
The  Bochner theorem (see Theorem \ref{thm:Bochner}) 
 ensures that the set of {\it characteristic functions} coincides with the 
set of 
{\it positive definite  functions} that are {\it  continuous.}

The main interest in working with this more general framework of functions is 
that we can easily derive a
nice estimate of the quantity $\vf(\xi^+)\vf(\xi^-)-\vf(\xi)$  
(see inequality (\ref{ineq:2})) 
that will be useful in the study of the collision term 
$Q$ with a non-integrable collision kernel.

Inspired by the papers of Toscani and his coauthors, we  introduce 
for each $\alpha\in [0,2]$ the  space 
\begin{equation}\label{PDa}
\PDa=\Big\{ \vf:\mathbb{R}^3\to \mathbb{C}\ \mbox{is a characteristic function 
such that} \;\;\|\vf-1\|_\alpha< \infty\Big\},
\end{equation}
where 
\begin{equation}\label{norm}
\|\vf-1\|_\alpha\equiv 
\sup_{\xi\in\RR} \frac{|\vf(\xi)-1|}{|\xi|^{\alpha}}.
\end{equation}
 The set $\PDa$ endowed with the distance
\begin{equation}\label{distance}
\|\vf-\widetilde\vf\|_\alpha 
\equiv \sup_{\xi\in\RR} \frac{|\vf(\xi)-\widetilde\vf(\xi)|}{|\xi|^\alpha}.
\end{equation}
is a complete metric space (see Proposition \ref{prop:PDa}, below).

The definition of $\PDa$ makes 
sense also for $\alpha>2$,
however, as we will see later, $\PDa=\{1\}$ in this case. 
In fact, in order to have a non trivial  function space in the case  $\alpha>2$, 
higher order 
moments should be considered and a suitable Taylor 
polynomial should be subtracted from $\vf$ in the definition of the space 
given by  eq. (\ref{PDa}) as it was done in \cite{CGT} 
 (see also \cite {V}).
On the other hand, $\mathcal{K}^0$ coincides with the set of all characteristic 
functions
and  the following imbeddings  hold true 
\begin{equation}\label{imb}
\{1\}\subseteq \mathcal{K}^{\alpha}\subseteq
\mathcal{K}^{\alpha_0}\subseteq \mathcal{K}^{0}
\quad \mbox{for all}\quad  
2\geq \alpha\geq \alpha_0\geq 0,
\end{equation}
see Lemma \ref{lem:imb} for the proofs of all these properties.

Let us also emphasize that the Fourier transform 
of any probability measure with the finite moment of order $\alpha$ 
belongs to $ \mathcal{K}^{\alpha}$.
This important feature, proved below in  
 Lemma \ref{lem:moment},
allows us to transfer the properties stated in  eq.~(\ref {as:f0}) 
for the function $f_0$ into properties 
to be verified by the new variable $ \vf_0$, justifying in this way the choice 
of the functional setting  $\PDa$. We show, however, that the set $\PDa$ is, 
in fact,
bigger than the set of the Fourier transforms 
of  probability measures with the finite moment of order $\alpha$, 
see Remark \ref{rem:counter}.

In the next section, we  present examples of functions from $\PDa$ as well 
as
some fundamental properties of the metric space $\PDa$.

Next, for every $\xi \in \RR\setminus \{0\}$, we define 
 the  quantity which appears systematically in our considerations:
\begin{equation}
\lambda_\alpha \equiv
 \int_{S^{2}} \B\left(\frac{\xi\cdot \sigma}{|\xi|}\right) 
\left( \frac{|\xi^-|^\alpha +|\xi^+|^\alpha}{|\xi|^\alpha }-1\right)\dn.
\label{la0}
\end{equation}
Note that, in view of relations \rf{xipm1}, we have $\lambda_2=0$.
In Corollary \ref{cor:la}, we prove that $\lambda_\alpha$
is finite, independent of $\xi$, and positive for $0<\alpha<2$,
under the assumption
$
(1-s)^{\alpha/2}(1+s)^{\alpha/2} \B(s)\in L^1(-1,1).
$
~However, to construct solutions to the initial-value problem \rf{Feq}--
\rf{Fini},
we have to impose the stronger assumption on the collision kernel, namely,
\begin{equation}\label{non-cut}
(1-s)^{\alpha_0/4}(1+s)^{\alpha_0/4} \B(s)\in L^1(-1,1) 
\quad\mbox{for some}\quad \alpha_0\in [0,2]. 
\end{equation}

\begin{rmk}\label{rem:kernel}
We have already mentioned in the introduction that 
the physical collision kernel $\B=\B(s)$ behaves 
at $s=1 $ as the function $(1-s)^{-5/4}$. 
Hence, the  assumption \rf{non-cut} holds true 
for this kind of singularity
if $-5/4+\alpha_0/4>-1$, that is for $\alpha_0>1$. 
More generally, as emphasized {\it e.g.}~in \cite{Ukai}
and in 
\cite[Ch.~1.1]{V}, 
there are important collision kernels
in physics and in modeling with the behavior $\B(s)\sim (1-s)^{-1-\nu}$
as $s\to 1$ for some $\nu>0$. We can  deal with this kind of singularity 
if $\nu<1/2$ provided $\alpha_0>4\nu$.
\end{rmk}

We are now in a position to state our main
 result  the existence of solutions to 
the initial value problem
\rf{Feq}--\rf{Fini}.

\begin{theorem}[Existence and uniqueness of solutions]\label{thm:exist}
Assume that $\B$ satisfies  assu\-mp\-tion
\rf{non-cut} for some $\alpha_0\in [0,2]$. 
Then  for each 
$\alpha \in [\alpha_0,2]$ and every $\vf_0\in \PDa$
 there exists 
 a  classical solution $\vf\in C([0,\infty), \PDa)$
of problem
\rf{Feq}--\rf{Fini}. The solution 
is unique in the space 
$C([0,\infty), \mathcal{K}^{\alpha_0})$.
\end{theorem}

Notice that, for every initial datum $\vf_0 \in \PDa$ with 
$\alpha\in [\alpha_0,2]$, 
the corresponding solution 
belongs to the space $C([0,\infty), \mathcal{K}^{\alpha_0})$ in view of  
imbedding~\rf{imb}.

\begin{rmk}
Let us first explain that Theorem \ref{thm:exist} generalizes known results on {\it finite energy} solutions
to the initial value problem \rf{eq}--\rf{ini}. Indeed, 
if $\vf_0\in \mathcal{K}^2$ 
is the Fourier transform of the function $f_0$ satisfying \rf{as:f0}, then 
the corresponding solution $\vf=\vf(\xi,t)$ of problem \rf{Feq}--\rf{Fini},
constructed in Theorem \ref{thm:exist}, is the Fourier transform 
of the solution $f=f(v,t)$  to the original initial value problem 
\rf{eq}-\rf{ini} which satisfies the important conservation laws from 
\rf{as:f1}.
To show this {\it persistence property}, it suffices to note that
the existence of the solution to \rf{eq}--\rf{ini} satisfying \rf{as:f1}
is well-known (see {\it e.g.} \cite{PT96} where the same argument is valid
for more general collision kernel satisfying \rf{non-cut}).
By uniqueness,  this solution agrees with our solution constructed
in Theorem \ref{thm:exist} in the space $C([0,\infty), \mathcal{K}^2)$.
\end{rmk}

\begin{rmk}
In Theorem \ref{thm:exist}, we construct a large class of smooth solutions 
(and not only probability measures) 
to the original initial value problem \rf{eq}--\rf{ini}.  
To see it, it suffices to apply  the well-known regularization procedure 
based on the Bobylev identity 
\begin{equation}\label{Bob:id}
Q(g*M,f*M)=Q(g,f)*M,
\end{equation}
where $Q$ is the Boltzmann operator \rf{B} and $M$ denotes the Maxwellian
probability distribution. Identity \rf{Bob:id} results immediately after 
computing the Fourier transform of its both sides and using  the Bobylev form of $\widehat Q$ 
together with the equality
$|\xi^+|^2+|\xi^-|^2=|\xi|^2$. 

Now, let $\widehat M(\xi)=e^{-A|\xi|^2}$ for some $A>0$ and 
$\vf_0=\widehat \mu_0 \in \PDa$ for some probability measure $\mu_0$.
Denote by $\vf=\vf(\xi,t)$ the solution to \rf{Feq}--\rf{Fini} with $\vf_0$ 
as the initial datum. 
By the Bobylev identity \rf{Bob:id} written in the Fourier variables, 
the function $\vf(\xi,t) e^{-A|\xi|^2}$ is 
the solution of 
problem \rf{Feq}--\rf{Fini} corresponding to the initial datum
$\vf_0\widehat M=\widehat{(\mu_0*M)} \in \PDa$.
Computing the inverse Fourier transform of this rapidely decreasing in $\xi$ 
solution to \rf{Feq}--\rf{Fini},
we obtain 
the smooth solution of the original problem
\rf{eq}--\rf{ini} with the initial condition $\mu_0*M$.

In this work, however, we do not address questions on regularity of solutions
to the homogeneous Boltzmann equation for Maxwellian molecules.
We refer the reader to the recent works \cite{DFT, Ukai} and to references 
therein for  proofs of  smoothing properties of {\it finite energy} solutions 
(namely, those satisfying \rf{as:f1})
to 
\rf{eq}--\rf{ini} including the Gevrey smoothing and the Sobolev space 
regularity.
\end{rmk}

Next, we prove  the stability inequality for solutions to problem 
\rf{Feq}--\rf{Fini} which were constructed in Theorem \ref{thm:exist}.

\begin{theorem}[Stability of solutions] \label{thm:stab}
Assume that  
$\B$ satisfies \rf{non-cut} for some $\alpha_0\in [0,2]$.
Let $\alpha \in [\alpha_0,2]$ and consider two solutions
$\vf,\wvf \in C ([0,\infty , \PDa)$ of the problem
\rf{Feq}--\rf{Fini} corresponding to the initial data $\vf_0,\wvf_0\in
\PDa$, respectively.
Then for every $t\geq 0$
\begin{equation}\label{stab:in:0}
\|\vf(t)-\wvf(t)\|_{\alpha}
\leq e^{\lambda_\alpha t} \|\vf_0-\wvf_0\|_{\alpha},
\end{equation}
where the constant $\lambda_\alpha\geq 0$ is defined in \rf{la0}.
\end{theorem}

The exponential growth in time on the 
right-hand-side of inequality \rf{stab:in:0} is optimal, see Remark
 \ref{rem:opt} below.
We  use this exponential estimate
in our study of the  asymptotic stability of solutions to 
problem \rf{Feq}--\rf{Fini}.

To prove Theorems \ref{thm:exist} and \ref{thm:stab},  
we begin by imposing the cut-off 
assumption on the kernel $\B$, namely, we assume  $\B\in L^1(-1,1)$
(this is the condition \rf{non-cut} with $\alpha_0=0$).
In this particular case, the results on the existence and the uniqueness of 
solutions to \rf{Feq}--\rf{Fini}
are not new. It is
well-known that, for integrable collision kernels,  the solution
of the initial value problem \rf{Feq}--\rf{Fini} has the explicit representation
via the Wild sum (see \rf{w:1}--\rf{w:2}, below) which is convergent under 
relatively weak assumptions imposed on the initial datum $\vf_0$ 
({\it cf.~e.g.}~\cite[Thm.~2.1]{PT96}).
Here, for the completeness of the exposition,
we prove that the Wild series converges in the space 
$C([0,\infty), \PDa)$. Moreover, 
we present another 
construction of solutions to \rf{Feq}--\rf{Fini}
under the cut-off condition imposed on $\B$,
based on the Banach 
contraction principle, see Theorem \ref{thm:exist:cut} in 
Section \ref{sec:exist}.

The proofs of the existence, 
the uniqueness, and the stability  
of solutions to \rf{Feq}--\rf{Fini} in the space 
$C([0,\infty), \PDa)$
in the case of nonintegrable collision kernels $\B$ satisfying 
\rf{non-cut}
are our main contribution to this theory.
We are able to remove the cut-off assumption
and to complete the proofs of  Theorems 
\ref{thm:exist} and~\ref{thm:stab} by using a well-known approximation argument 
combined with suitable (and crucial for our reasoning) estimates
for characteristic functions form the space $\PDa$, see Lemma \ref{lem:in:c}.

Next, we study the large time behavior of solutions 
to the initial value problem 
\rf{Feq}--\rf{Fini}.
Here, the key role is played by self-similar solutions of equation 
\rf{Feq}
constructed by  Bobylev and Cercignani 
\cite{BC02-eternal,BC02-self}
in the following form
\begin{equation}\label{self-sol}
\vf(\xi,t)= \VF(\xi e^{\mu t})\quad\mbox{for some}\quad \mu \in \mathbb{R}.
\end{equation}
Substituting the function $\vf$ from \rf{self-sol}
into equation \rf{Feq}, we obtain the equation
for the profile $\VF$ (here, $\eta=\xi e^{\lambda t}$)
\begin{equation}\label{Feq-self}
\mu \eta \cdot \nabla \VF(\eta) =
\int_{S^2} \B\left(\frac{\eta\cdot \sigma}{|\eta|}\right)
\big(\VF(\eta^+)\VF(\eta^-)-\VF(\eta)\VF(0)\big)\dn,
\end{equation}
where $\eta^+$ and $\eta^-$ are defined analogously as the vectors in 
\rf{xipm}.

Below, in Lemma \ref{lem:lambda-alpha}, we recall an argument which allows
us to calculate the scaling parameter $\mu$. We show that if a radial 
solution  $\VF\in\PDa$ of equation \rf{Feq-self} satisfies
$
\lim_{|\eta|\to 0}\big(\VF(\eta)-1\big)|\eta|^{-\alpha}=K
$ 
for some constant $K\neq 0$ then, necessarily,
\begin{equation}\label{mua0}
\mu=\mu_\alpha=\frac{\lambda_\alpha}\alpha,
\end{equation}
where the constant $\lambda_\alpha$ is defined in \rf{la0}.
This is a well-founded argument because, for every $\alpha\in (0,2)$ and $K\neq 0$,
Bobylev and Cercignani \cite{BC02-self} proved the existence of 
 a solution  $\VF=\VF_{\alpha,K}$  to equation \rf{Feq-self} 
satisfying
\begin{equation}\label{self:prop:0}
\lim_{|\eta|\to 0}\frac{\VF_{\alpha,K}(\eta)-1}{|\eta|^\alpha}=K.
\end{equation}
In  Theorem \ref{thm:BC:self} below, we sketch the Bobylev and Cercignani 
construction. Here, we only
notice  that the constant $K$ in \rf{self:prop:0}
has to be nonpositive 
because every characteristic function $\VF_{\alpha,K}$
satisfies $|\VF_{\alpha,K}(\eta)|\leq 1$ for all $\eta\in\RR$, 
see Remark \ref{rmk:K<0} for more details.

\begin{rmk}\label{rem:self}
If the solution $\vf(\xi,t)$ and the self-similar profile $\VF(\xi)$
from \rf{self-sol}
are the Fourier transforms of functions $f=f(v,t)$ and $F=F(v)$, respectively,
then we obtain the self-similar solution of the Boltzmann equation 
\rf{eq}--\rf{B} in the original variables in the form
$$
f(v,t)=e^{-3\mu t}F(ve^{-\mu t}).
$$ 
Obviously, $\int_\RR f(v,t)\,dv=\int_\RR F(v)\,dv$ for all $t\in \R$. 
This solution, however, 
cannot have  finite energy, because  the condition
$F\in L^1(\RR, |v|^2\,dv)$ leads immediately to the equality
$$
\int_\RR f(v,t)|v|^2\,dv=e^{2\mu t}\int_\RR F(v)|v|^2\,dv
$$  
which contradicts \rf{as:f1} if $\mu \neq 0$, see 
\cite{BC02-eternal, BC02-self} for more detailed discussion.
\end{rmk}

In order to  study the asymptotic stability of the self-similar
solutions $\vf(\xi,t)=\VF_{\alpha,K}(\xi e^{\mu_\alpha t})$ as well as
 the large time behavior of other solutions to 
system \rf{Feq}--\rf{Fini},
it is more convenient to work in self-similar variables.
Hence, given a solution $\vf=\vf(\xi,t)$ to equation \rf{Feq}
we consider the new function
\begin{equation}\label{phi=psi}
\psi(\xi,t)=\vf (\xi e^{-{\mu_\alpha t}} ,t)\quad \mbox{with} \quad \mu_\alpha=\frac{\lambda_\alpha}{\alpha},
\end{equation}
which is the solution of the initial value problem
\begin{align}\label{Feq:psi}
\partial_t\psi+\mu_\alpha \xi\cdot\nabla \psi
&= \int_{S^2} \B\left(\frac{\xi\cdot \sigma}{|\xi|}\right)
\big(\psi(\xi^+,t)\psi(\xi^-,t)-\psi(\xi,t)\psi(0,t)\big)\dn,\\
\psi(\xi, 0)&=\psi_0(\xi)=\vf_0(\xi). \label{Fini:psi}
\end{align}
Note that, in the new variables, the self-similar profiles $\VF_{\alpha,K}$ 
are stationary solutions of equation \rf{Feq:psi}
({\it cf.}~equation \rf{Feq-self}).

Now, we are in a position to state our main result on the large 
time asymptotics
of solutions to \rf{Feq}--\rf{Fini}. 

\begin{theorem}[Large time asymptotics of solutions] \label{thm:asymp}
Assume that the collision kernel $\B$ satisfies the non cut-off condition
\rf{non-cut} for some $\alpha_0 \in (0,2)$.
Let $\alpha \in [\alpha_0,2)$.
Suppose that $\psi_0,\widetilde\psi_0\in \PDa$ satisfy
\begin{equation}\label{psi:0}
\lim_{|\xi|\to 0}\frac{\psi_0(\xi)-\widetilde\psi_0(\xi)}{|\xi|^\alpha} =0.
\end{equation}
Then the corresponding solutions $\psi(\xi,t)$ and $\widetilde \psi(\xi,t)$ of 
the rescaled problem
\rf{Feq:psi}--\rf{Fini:psi} approach each other 
in the following sense
$$
\lim_{t\to\infty}\|\psi(t)-\widetilde\psi(t) \|_\alpha =0.
$$
\end{theorem}

The proof of Theorem \ref{thm:asymp} (given in Section \ref{sec:asymp})
 is very simple and
is an almost immediate consequence of the generalized version of 
the stability inequality \rf{stab:in:0} (see  Lemma  \ref{lem:stab}
and Corollary \ref{cor:stab:non-cut}, below).

Combining  Theorem \ref{thm:asymp} with the property of the self-similar profile
stated in \rf{self:prop:0}, we
find   the condition on the
initial datum $\psi_0=\vf_0$ such the corresponding solution
of \rf{Feq}--\rf{Fini} converges (in self-similar variables) toward
the self-similar profile $\VF_{\alpha,K}$. This particular case of Theorem 
\ref{thm:asymp} is stated in the following corollary.

\begin{corollary}[Self-similar asymptotics]\label{cor:asymp:self}
Assume that the collision kernel $\B$ satisfies 
\rf{non-cut} for some $\alpha_0 \in (0,2)$. Let $\alpha \in [\alpha_0,2)$.
Consider the initial datum  $\psi_0\in \PDa$ such that
\begin{equation}\label{psi:0:K}
\lim_{|\xi|\to 0}\frac{\psi_0(\xi)-1}{|\xi|^\alpha} =K
\quad\mbox{for some} \quad K\leq 0.
\end{equation}
Denote by $\VF_{\alpha,K}$ the self-similar profile of Bobylev and Cercignani.
Then the corresponding solution $\psi(\xi,t)$ of problem
\rf{Feq:psi}--\rf{Fini:psi}
satisfies
$$
\lim_{t\to\infty}\|\psi(t)-\VF_{\alpha,K}\|_\alpha =0
\quad \mbox{if} \quad K<0
$$
and
$$
\lim_{t\to\infty}\|\psi(t)-1\|_\alpha =0
\quad \mbox{if} \quad K=0.
$$
\end{corollary}

\begin{rmk}
It is worth to reformulate the above asymptotic results for solutions to 
problem \rf{Feq}--\rf{Fini} before rescaling stated in \rf{phi=psi}.
Under the assumptions of Theorem~\ref{thm:asymp}, for every initial conditions
$\vf_0,\wvf_0\in\PDa$ such that
$$
\lim_{|\xi|\to 0}\frac{\vf_0(\xi)-\wvf_0(\xi)}{|\xi|^\alpha}=0,
$$ 
 the corresponding solutions $\vf=\vf(\xi,t)$ anf $\wvf=\wvf(\xi,t)$
of problem \rf{Feq}--\rf{Fini} satisfy
\begin{equation*}
\lim_{t\to\infty} e^{-\lambda_\alpha  t} \|\vf(t)-\wvf(t)\|_\alpha=0.
\end{equation*}
This is the immediate consequence of the change of variables \rf{phi=psi}
leading to the following equalities
$$ 
\|\psi(t)-\widetilde\psi(t)\|_\alpha
=
\sup_{\xi\in\RR}\frac{|\vf(\xi e^{-\mu_\alpha t},t)-\wvf(\xi e^{-\mu_\alpha t},t)|}{|\xi|^\alpha}
=
e^{-\lambda_\alpha  t} \|\vf(t)-\wvf(t)\|_\alpha
$$
due to the identity $\lambda_\alpha=\alpha\mu_\alpha$.
\end{rmk}

\begin{rmk}\label{rem:Max}
Let us discuss the large time behavior of solutions to \rf{Feq}--\rf{Fini}
in the case  $\alpha=2$.
Note first that $\mu_2=\lambda_2=0$ which is the immediate consequence 
of the definitions \rf{la0} and \rf{mua0} combined with 
the second equality in \rf{xipm1}.
The proof of Theorem \ref{thm:asymp} does not work for $\alpha=2$
because 
the assumption
$\lambda_\alpha\neq 0$  is essential in our reasoning.
In particular, we are not able to adapt the proof of Theorem 
\ref{thm:asymp} to show the convergence of solutions to \rf{Feq}--\rf{Fini}
from the space $\mathcal{K}^2$ toward  Maxwellians in the Fourier variables,
$\VF_A(\eta)=e^{-A|\eta|^2}$ with $A>0$, which are the solutions of equation 
\rf{Feq-self}
with $\mu=0$. 

The large time asymptotics of solutions
to \rf{Feq}--\rf{Fini} in this limit case $\alpha=2$
was studied by Toscani and Villani \cite{TV}. 
Their result on the convergence of 
solutions to \rf{Feq}--\rf{Fini} to Maxwellians $\VF_A(\eta)=e^{-A|\eta|^2}$
stated in \cite[Cor. 5.3]{TV} 
can be formulated as follows. Assume that $\vf_0$ is the Fourier transform of
a mean zero probability measure with finite second moment 
(hence, $\psi_0\in \mathcal{K}^2$ by Lemma \ref{lem:moment}) which, 
moreover, satisfies \rf{psi:0:K} for some $K=-A<0$. 
Denote by $\psi=\psi(\xi,t)$ 
the corresponding solution.
Then $\|\psi(t)-\VF_A\|_2$ is decreasing to 0 as $t\to \infty$.
\end{rmk}

In view of Remark \ref{rem:Max}, the asymptotics  stated in Corollary 
\ref{cor:asymp:self} should be treated as the extension of the classical
result on the convergence of solutions of the Boltzmann equation to Maxwellian.
Roughly speaking, Corollary \ref{cor:asymp:self} says that
solutions to original problem \rf{eq}--\rf{ini} with infinite energy 
({\it i.e.}~when the Fourier transform of their initial conditions satisfy
\rf{psi:0:K}  for some $\alpha \neq 2$ and $K\neq 0$)
converge, in self-similar variables, toward 
the universal probability measure which is 
the Fourier transform of the self-similar profile
  constructed by Bobylev and Cercignani. This probability measure should 
be treated as the counterpart of Maxwellian which is the solution of 
\rf{Feq-self} with $\alpha=2$ (recall that $\mu_2=0$).
Such a convergence, in a pointwise sense and for very particular  
 initial conditions (radially symmetric and in the form of a series) was  
proved in \cite[Thm.~6.2]{BC02-self}. 
The very simple proof of the convergence
of solutions to \rf{Feq}--\rf{Fini}
toward self-similar profile $\Phi_{\alpha,K}$
in the metric $\|\cdot\|_\alpha$ for every solution (not necessarily radial)
corresponding to 
the initial datum $\psi_0\in\PDa$ satisfying \rf{psi:0:K} is our main contribution 
to this theory.

\section{Continuous positive definite functions}

Since we deal with the Fourier transform of the Boltzmann equation 
and since this equation describes the time evolution of a probability 
measure (the unknown function $f(v,t)$) it is natural to begin our investigation 
recalling the classical definition of  ``characteristic functions". These 
functions have been systematically 
used in the papers  devoted to the study of the 
homogeneous Boltzmann equation  \rf{Feq}--\rf{Fini} in Fourier  variables (e.g. 
\cite{B1,B2,BC02-eternal, BC02-self, CGT, D03, DFT, GTW, TV,V}). 

\begin{definition}\label{def:CF}
A function $\vf:\R^N\to \C$ 
is called  characteristic function if there is a probability measure $ \mu$ 
(i.e. a  Borel measure 
with $ \int_{\mathbb{R}^N}\mu(dx)=1$) such that 
 we have the following identity 
$ \vf(\xi)=\widehat\mu(\xi) =\int_{\R^N} e^{-ix\cdot \xi}\,\mu(dx)$. 
The set of all characteristic functions $\vf:\R^N\to\C$ we 
will be denoted by ${\mathcal{K}}$.
\end{definition}

In some estimates presented in this section, it is more convenient to
introduced 
the more general setting provided by the  definition of 
a {\it positive definite function}.

\begin{definition}\label{def:pos:def}
A function $\vf:\R^N\to \C$ 
is called  positive definite 
if for every $k\in\N$ and every vectors $\xi^1, ...,\xi^k\in \R^N$ the matrix
$\big(\vf(\xi^j-\xi^\ell)\big)_{j,\ell=1,...,k}$ is positive Hermitian, 
{\it i.e.} for all
$\lambda_1,...,\lambda_k\in\C$ w have
\begin{equation}\label{pos:def}
\sum_{j,\ell=1}^k \vf(\xi^j-\xi^\ell)\lambda_j\overline{\lambda_\ell}\geq 0.
\end{equation}

\end{definition}

The following celebrated theorem by Bochner plays a fundamental role in the 
theory of
positive definite functions, since it states that the set of {\it {continuous}} 
positive 
definite functions coincides with the set of characteristic functions. 

\begin{theorem}\label{thm:Bochner}
A function $\vf:\mathbb{R}^N\to \mathbb{C}$ is a characteristic function if and 
only 
if the 
following conditions are fulfilled
\begin{itemize}
\item[i.] $\vf$ is a continuous function on $\mathbb{R}^N$
\item[ii.] $\vf(0)=1$
\item[iii.] $\vf$ is {\it positive definite}.
\end{itemize}
\end{theorem}

We refer the reader to the books
either by Berg and Forst \cite[Ch.~I, \S 3]{BF75} or 
 by Jacob \cite[Ch.~3]{J1} 
for proofs  of properties 
of positive definite 
functions which will be listed below.

The reason why we prefer to introduce the larger set of positive definite 
functions  (instead of simple  
characteristic functions) is that we can easily derive estimates on a certain
product 
of positive definite  functions (see inequality (\ref{ineq:2})) that will be useful for the study of 
the 
collision operator. 

Before deriving such key estimates, we start with much simpler results that 
follow immediately from the definition of  positive definite functions.

\begin{lemma}
Every positive definite function $\vf$ 
satisfies
\begin{equation}\label{vf:1}
\overline{\vf(\xi)} =\vf(-\xi) \quad \mbox{and}\quad \vf(0)\geq 0
\end{equation} 
and
\begin{equation}\label{vf:2}
|\vf(\xi)|\leq \vf(0), \quad \mbox{hence}\quad 
\sup_{\xi\in\mathbb{R}^N}|\vf(\xi)|=\vf(0).
\end{equation}
\end{lemma}

\begin{lemma}\label{lem:comb}
Any linear combination with positive coefficients 
of positive definite functions
is a positive definite function. 
The set of  positive definite functions is closed with respect
to the pointwise convergence.
\end{lemma}

\begin{lemma}
The product of two positive definite functions is a positive definite function.
\end{lemma}
\begin{proof}
This is the immediate consequence of 
Definition  \ref{def:pos:def} if we note that for every 
two positive Hermitian  matrices  $(a_{jk})_{j,k=1,...,N}$ and
 $(b_{jk})_{j,k=1,...,N}$, the matrix
 $(c_{jk})_{j,k=1,...,N}$ with elements $c_{jk}=a_{jk}b_{jk}$
is positive Hermitian, see {\it e.g.} \cite[Lemma 3.5.9]{J1}.
\end{proof}

\begin{lemma}\label{lem:Re}
If $\vf$ is a positive definite function, so are $\overline\vf$ and $\Re \vf$.
\end{lemma}
\begin{proof}
To show that $\overline\vf$ is a positive definite function it suffices to 
compute the complex conjugate of inequality \rf{pos:def}. Using equality $\Re 
\vf =(\vf+\overline\vf)/2$ we complete the proof by Lemma \ref{lem:comb}.
\end{proof}

Now, we state two important inequalities for positive definite functions 
which  play the fundamental role in our reasoning when we deal with 
nonintegrable collision kernels. 
By this reason, for the completeness of the exposition, we sketch their proofs,
 see either \cite[Ch.~I, \S 3.4]{BF75} or 
\cite[Lemma 3.5.10]{J1} for more details.

\begin{lemma}
For any positive definite function $\vf=\vf(\xi)$ such that $\vf(0)=1$ we have
\begin{equation}\label{ineq:1}
|\vf(\xi)-\vf(\eta)|^2\leq 2 \big(1-\Re \vf(\xi-\eta)\big)
\end{equation}
and
\begin{equation}\label{ineq:2}
|\vf(\xi)\vf(\eta)-\vf(\xi+\eta)|^2\leq (1-|\vf(\xi)|^2)(1-|\vf(\eta)|^2)
\end{equation}
for all $\xi,\eta\in \R^N$.
\end{lemma}

\begin{proof} 
We are going to use 
inequality \rf{pos:def}
with  suitable chosen vectors $\xi^j$ and constants~$\lambda_j$. 
Indeed, for $\xi,\eta\in \R^N$ such that $\vf(\xi)\neq \vf(\eta)$ 
we consider the Hermitian
matrix
\begin{equation}\label{matrix}
\left(
\begin{array}{ccc}
\vf(0)&\overline{\vf(\xi)}&\overline{\vf(\eta)}\\
\vf(\xi)& \vf(0)& \vf(\xi-\eta)\\
\vf(\eta)& \overline{\vf(\xi-\eta)}&\vf(0)
\end{array}
\right),
\end{equation}
where $\vf(0)=1$.
Next, with arbitrary and given $s\in\R$, we 
define
$$
\lambda_1=s,\quad \lambda_2=\frac{s|\vf(\xi)-\vf(\eta)|}{\vf(\xi)-\vf(\eta)},
\quad \lambda_3=-\lambda_2.
$$
Hence, applying inequality \rf{pos:def}, 
we find by a straightforward calculation
$$
1+2s^2+2s|\vf(\xi)-\vf(\eta)|-2s^2\Re \vf(\xi-\eta)\geq 0.
$$
This means that the discriminate of the quadratic form on the 
left-hand side  (as the function of 
$s$) has to be nonpositive, hence,
$$
4|\vf(\xi)-\vf(\eta)|^2\leq 4(2-2\Re \vf(\xi-\eta)).
$$
which completes the proof of \rf{ineq:1}.

On the other hand,
inequality \rf{ineq:2} is equivalent to the fact that the determinant of the
Hermitian matrix \rf{matrix} with $\vf(0)=1$ is non-negative.
\end{proof}

Let us now recall the definition of the function space  

\begin{equation}\label{PDa2}
\PDa=\Big\{ \vf:\mathbb{R}^3\to \mathbb{C}\ \mbox{is a characteristic function 
such that} \;\;\|\vf-1\|_\alpha< \infty\Big\},
\end{equation}
supplemented with the metric
\begin{equation*}
\|\vf-\widetilde\vf\|_\alpha 
\equiv \sup_{\xi\in\RR} \frac{|\vf(\xi)-\widetilde\vf(\xi)|}{|\xi|^\alpha}.
\end{equation*}

First, we give  examples of characteristic functions from
the space $\PDa$.

\begin{examp}\label{exa}
\begin{itemize}

\item[i.] The function $\vf=\vf(\xi)$ satisfying $\vf(0)=1$ and $\vf(\xi)=0$ for 
$\xi$ different from zero is a positive 
definite function, however,
 it is not a characteristic function (since it is not 
continuous)

\item[ii.] The function $\vf(\xi)=e^{-ib\cdot\xi}$, with fixed $b\in\RR$,
 is the Fourier transform
of the Dirac delta $\delta_b$ concentrated at $b$. It belongs to $\PDa$ for 
every $\alpha\in [0,1]$.

\item[iii.] Maxwellians in the Fourier variables, $\vf(\xi)=e^{-A|\xi|^2}$ with 
fixed $A>0$, belongs to $\PDa$ for every $\alpha\in [0,2]$.

\item[iv.] The function $\vf_\alpha(\xi)=e^{-|\xi|^\alpha}$ is a 
characteristic function for each $\alpha\in (0,2]$ because this is the Fourier
transform of the probability distribution of an $\alpha$-stable symmetric 
L\'evy process, see {\it e.g.} \cite[Examples 3.5.23 and 3.9.17]{J1}  for more 
details.
Hence, $\vf_\alpha\in\mathcal{K}^{\beta}$ for 
each $\beta\in [0,\alpha]$.
\end{itemize}
\end{examp}

\begin{proposition}\label{prop:PDa}
For every $\alpha\in [0,2]$, the set $\PDa$
endowed with the distance \rf{distance}
is a complete metric space.
\end{proposition}
 
\begin{proof}
The proof is immediate because the set of characteristic functions is
closed with respect to the pointwise convergence. 
\end{proof}

Next, we state without the proof simple properties 
of the  space $\PDa$.

\begin{lemma}\label{lem:elem}
\begin{itemize}
\item[i.] The space $\PDa$ is not a vector space ({\it e.g.} 
$\vf(\xi)\equiv 0$ does not belong to $\PDa$).

\item[ii.] $\vf\equiv 1\in\PDa$ for every $\alpha\geq 0$.
\item[iii.] For every $\vf\in\PDa$ we have
$
 |\vf(\xi)|\leq \vf(0)=1
$
({\it cf.} \rf{vf:2}).

\item[iv.] For all $\vf,\wvf\in\PDa$ their product satisfies 
$
 \vf \wvf\in\PDa.
$

\item[v.] Any linear and convex  combination of functions from 
$\PDa$ belongs to $\PDa$ ({\it cf.} Lemma \ref{lem:comb}). 

\end{itemize}
\end{lemma}

In the following lemma, we explain why we limit ourselves to 
$\alpha\in [0,2]$ in the definition of $\PDa$.

\begin{lemma}\label{lem:imb}
\begin{itemize}
\item[i.] ${\mathcal K}^0={\mathcal K}$
\item[ii.] ${\mathcal K}^{\alpha_1}\subseteq{\mathcal K}^{\alpha_2}$ if 
$\alpha_2\leq \alpha_1$.
\item[iii.] ${\mathcal K}^\alpha=\{1\}$ for every $\alpha>2$.
\end{itemize}
\end{lemma}

\begin{proof}
In the case of i., it suffices to use \rf{vf:2} in order to see that any 
characteristic  function  
$\vf$ is bounded, more precisely, it 
satisfies
$\sup_{\xi\in\R^N}|\vf(\xi)-1|\leq \vf(0)+1$.

To show ii., for any $\vf\in{\mathcal K}^{\alpha_1}$, we proceed as 
follows
\begin{equation*}
\begin{split}
\|\vf-1\|_{\alpha_2}&
\leq \sup_{|\xi|\leq 1}\frac{|\vf(\xi)-1|}{|\xi|^{\alpha_2}}
+\sup_{|\xi|> 1}\frac{|\vf(\xi)-1|}{|\xi|^{\alpha_2}}\\
&
\leq \sup_{|\xi|\leq 1}\frac{|\vf(\xi)-1|}{|\xi|^{\alpha_1}}
+\sup_{|\xi|> 1}|\vf(\xi)-1|\\
&\leq \|\vf -1\|_{\alpha_1}+\vf(0)+1,
\end{split}
\end{equation*}
since 
$\alpha_2\leq \alpha_1$ and by using \rf{vf:2}. Hence, $\vf \in {\mathcal 
K}^{\alpha_2}$.

Let us show  iii. It follows immediately form 
eq. \rf{norm} that any $\vf \in \PDa$ with $\alpha >2$ satisfies
\begin{equation}\label{iii}
\left|\frac{1-\vf(\xi)}{|\xi|^2}\right|\leq |\xi|^{\alpha-2}\|\vf-1\|_\alpha
\to 0 \quad\mbox{as}\quad |\xi|\to 0.
\end{equation}
Next, using inequality \rf{ineq:1} we get for any unit vector $\zeta\in \RR$ 
and all $\xi\in\RR$ 
\begin{equation*}
\left|\frac{\vf(\xi+h\zeta)-\vf(\xi)}{h}
\right|^2\leq
2\frac{\big(1-\Re\vf (h\zeta)\big)}{h^2}
\leq 2\left| \frac{1-\vf (h\zeta)}{h^2}\right|,
\end{equation*}
thus, by \rf{iii}, we have
\begin{equation*}
\lim_{h\to 0}\frac{\vf(\xi+h\zeta)-\vf(\xi)}{h}=0.
\end{equation*}
Hence,  for all $\zeta\in \RR$ the directional derivative 
$\zeta\cdot\nabla \vf(\xi)$ 
exists and is equal to zero, implying that $\vf$ is constant. 
\end{proof}

\begin{lemma}\label{lem:ReIm}
Let $\alpha\in [0,2]$.
Assume that $\vf\in \PDa$. Then $\Re\vf \in\PDa$,
\begin{equation}\label{ReIm}
\|\Re\vf -1\|_\alpha\leq \|\vf-1\|_\alpha,
\quad\mbox{and}\quad
\sup_{\xi\in \RR\setminus\{0\}}\frac{|\Im \vf(\xi)|}{|\xi|^\alpha}\leq 
\|\vf-1\|_\alpha.
\end{equation}
\end{lemma}

\begin{proof}
Let $\vf\in \PDa$. It is well-known that $\Re\vf$ is a characteristic function
({\it e.g.}~it suffices to combine Lemma \ref{lem:Re} with the Bochner Theorem
\ref{thm:Bochner}).
Now, by the Pythagoras theorem, we obtain
\begin{equation}\label{ReIm:1}
|\vf(\xi)-1|^2=|\Im \vf(\xi)|^2+|\Re\vf(\xi)-1|^2
\geq |\Re\vf(\xi)-1|^2.
\end{equation}
Hence, we complete the proof of the first inequality in \rf{ReIm}
dividing \rf{ReIm:1} by $|\xi|^\alpha$ and computing 
the supremum with respect to $\xi\in\RR$.

To show the second inequality in \rf{ReIm}, 
we proceed analogously using the inequality
$|\vf(\xi)-1|\geq |\Im \vf(\xi)|$ resulting from \rf{ReIm:1}.
\end{proof}

Now, we are in a position to prove an inequality which 
implies (see the proof of Lemma \ref{lem:AA})
that the nonlinear term in equation \rf{Feq} is well-defined 
for functions from $\PDa$ if we impose the condition \rf{non-cut} on 
the collision kernel.

\begin{lemma}\label{lem:in:c}
Let $\alpha\in [0,2]$.
Assume that $\vf\in \PDa$. For every $\xi\in\RR$ define $\xi^+$ and $\xi^-$ by
equations \rf{xipm} with some fixed $n\in S^2$.
Then
\begin{equation}\label{vf:alpha}
|\vf(\xi^+)\vf(\xi^-)-\vf(\xi)|\leq 4 |\xi^+|^{\alpha/2}|\xi^-|^{\alpha/2}
\|\vf-1\|_\alpha.
\end{equation}
\end{lemma}

\begin{proof}
First, recall that $\vf(0)=1$.
We begin the elementary identity
\begin{equation}\label{1-vf2}
1-|\vf(\xi^+)|^2= (1-\vf(\xi^+))(1+\overline{\vf(\xi^+)})+ 2\,\Im \vf(\xi^+).
\end{equation}
Using the estimate $|1+\overline{\vf(\xi^+)}|\leq 1+|\vf(\xi^+)|\leq 2$
({\it cf.} \rf{vf:2}) and second inequality in \rf{ReIm}
we deduce from \rf{1-vf2} 
\begin{equation*}
0\leq 1-|\vf(\xi^+)|^2\leq 4|\xi^+|^\alpha\|\vf-1\|_\alpha.
\end{equation*}
Obviously, an analogous inequality holds true if we replace $\xi^+$ by $\xi^-$.
Now, applying inequality \rf{ineq:2}, we conclude
\begin{equation*}
\begin{split}
|\vf(\xi^+)\vf(\xi^-)-\vf(\xi)|&
\leq \sqrt{\big(1-|\vf(\xi^+)|^2\big)\big(1-|\vf(\xi^-)|^2\big)}\\
&\leq 4 |\xi^+|^{\alpha/2}|\xi^-|^{\alpha/2}
\|\vf-1\|_\alpha
\end{split}
\end{equation*}
for all $\xi\in \RR$.
\end{proof}

\begin{lemma}\label{lem:moment}
Let $\alpha\in [0,2]$.
Assume that $\mu$ is a probability measure on $\RR$ such that 
$\int_\RR |v|^\alpha\,\mu(dv)$ is finite.
If, moreover, $\alpha\in (1,2]$, assume that 
 $\int_{\RR}v_i \,\mu(dv)=0$ for  
$i\in \{1, 2,3\}$.
Then $\widehat \mu\in \PDa$.
\end{lemma}

\begin{proof}
Consider first $\alpha\in (0,1]$.
Using the definition of the Fourier transform 
of a probability measure $\mu(dv)$ 
 we obtain
\begin{equation}\label{mom:1}
\frac{|\widehat \mu(\xi)-1|}{|\xi|^\alpha}
\leq \int_{\R^3}
\frac{|e^{-i v\cdot\xi}-1|}{|\xi|^\alpha}\,\mu(dv).
\end{equation}
Note now the  by substituting $\xi=\eta/|v|$, we have
$$\sup_{\xi\in\R^3}\frac{|e^{-iv\cdot\xi}-1|}{|\xi|^{\alpha}}
=|v|^\alpha \sup_{\eta\in\R^3}
\frac{|e^{-i\eta\cdot v/|v|}-1|}{|\eta|^{\alpha}}\leq C|v|^\alpha,
$$ 
where, in view of the elementary inequality 
$|e^{is}-1|\leq |s|$ for all $s\in\R$, 
 the constant $C=\sup_{v,\eta\in \R^3} |e^{-i\eta \cdot v/|v|}-1||\eta|^{-
\alpha}$ is finite for $\alpha\in (0,1]$.
Hence, we  deduce
from \rf{mom:1} that 
$$
\|\widehat \mu-1\|_\alpha \leq C\int_{\R^3}|v|^\alpha \,\mu(dv),
$$

For $\alpha\in (1,2]$, one should proceed analogously using the following 
counterpart of inequality 
\rf{mom:1}
\begin{equation*}
\frac{|\widehat \mu(\xi)-1|}{|\xi|^\alpha}\leq \int_{\R^3}
\left|\frac{e^{-iv\cdot \xi}+iv\cdot \xi-1}{|\xi|^\alpha}\right|\,\mu(dv).
\end{equation*}
being the simple consequence of the additional assumption 
$\int_{\R^3}v_i 
\,\mu(dv)=0$, for every
$i=\{1,2,3\}$.
\end{proof}

\begin{rmk}\label{rem:counter}
Let us  provide a counterexample 
that the reverse implication in Lemma \ref{lem:moment} for $\alpha\in (0,2)$ 
is not true, in other 
words,
we want to show that the space $\PDa$ is bigger than the space of 
of characteristic functions corresponding to
probability 
measures with finite moments of order $\alpha.$
Is is well-known that the function 
$\vf_\alpha(\xi)=e^{-|\xi|^\alpha},$ with $\alpha\in (0,2)$,  
is the Fourier transform of 
the probability density $P_\alpha(x)$ 
of the $\alpha$-stable symmetric L\'evy process,
(see Example \ref{exa}).
Obviously, we have $\vf_\alpha\in\PDa$. 
On the other hand, it is  known that for every 
$\alpha\in (0,2)$ the function
$P_\alpha$
is  smooth, nonnegative, and
satisfies the estimate
$
0<P_\alpha(x)\le C(1+|x|)^{-(\alpha+n)} 
$
for a constant $C$ and all $x\in\R^n$.
Moreover,
\begin{equation}\label{as P}
\frac{P_\alpha(x)}{|x|^{\alpha+n}}\to c_0
\quad\hbox{when}\quad  |x|\to\infty,
\end{equation}
where
$ c_0=\alpha2^{\alpha-1}\pi^{-(n+2)/2}\sin
(\alpha\pi/2)\Gamma\Bigl(\frac{\alpha+n}{2}\Bigr)
\Gamma\Bigl(\frac{\alpha}{2}\Bigr).$
We refer the reader to~\cite{BluG60} for a proof of the  formula
(\ref{as P}) with the explicit constant $c_0$.

In view of the limit relation \rf{as P}, 
we have  $\int_\RR P_\alpha(x)|x|^\alpha\,dx=\infty$.
\end{rmk}



\section{Existence under  cut-off assumption}\label{sec:exist}

In this section, we  
construct solutions of
 the initial value problem 
\rf{Feq}--\rf{Fini} and we study their stability 
 in the space $\PDa$ 
 imposing the usual cut-off assumption on the
collision kernel $\B$ 
(the pseudo-Maxwellian gas), say :
\begin{equation}\label{cut-off}
 \int_{S^2} \B\left(\frac{\xi\cdot \sigma}{ |\xi|}\right)\dn
\quad \mbox{is finite for all}\quad \xi\in \RR\setminus \{0\}.
\end{equation}
In the next section, we show how  to relax this condition.

\subsection{Technical results on the collision kernel}

Let us first introduce parameters  which appear 
systematically in our reasoning below.

\begin{lemma} \label{lem:gamma}
Let $\alpha\in [0,2]$ and $\B\in L^1(-1,1)$. 
Then for all $\xi\in \RR\setminus\{0\}$
the following quantity
\begin{equation}\label{gamma}
\gamma_\alpha \equiv
\int_{S^2} 
\B\left(\frac{\xi\cdot \sigma}{|\xi|}\right) \frac{|\xi^{-
}|^\alpha+|\xi^{+}|^\alpha}{|\xi|^\alpha}\dn
\end{equation}
is finite and independent of $\xi$. Moreover,
\begin{equation}\label{a>b}
\gamma_\alpha>\gamma_2=2\pi\int_{-1}^1 \B(s)\,ds \quad \mbox{if}\quad 0<\alpha< 
2.
\end{equation}
\end{lemma}

\begin{proof} 
Let $\sigma=(\sigma_1,\sigma_2,\sigma_3)\in S^2$. 
Rotating $\RR$ (if necessary) and using
the spherical coordinates we obtain  the equalities
\begin{equation}\label{g}
\int_{S^2} g\left(\frac{\xi\cdot \sigma}{|\xi|}\right)\dn
=\int_{S^2} g(\sigma_3)\dn= 2\pi \int_{-1}^1 g(s)\,ds,
\end{equation}
valid for every $g\in L^1(-1,1)$ 
and $\xi \in \RR\setminus \{0\}$. Hence, 
by \rf{g} with $g=\B$,
recalling  relations
\rf{xipm1}  we have
$\gamma_2=2\pi\int_{-1}^1 \B(s)\,ds$.

For $0<\alpha< 2$, we rewrite equalities \rf{xipm2} as follows
\begin{equation}\label{xipm:alpha}
|\xi^+|^\alpha=|\xi|^\alpha \left(\frac{1+\frac{\xi}{|\xi|}\cdot 
\sigma}2\right)^{\alpha/2}
\quad \mbox{and}\quad 
|\xi^-|^\alpha=|\xi|^\alpha\left(\frac{1-\frac{\xi}{|\xi|}\cdot 
\sigma}2\right)^{\alpha/2}
\end{equation}
and we use equality \rf{g} with 
$g(s)= \B(s)\left[\left(\frac{1+s}2\right)^{\alpha/2}+
\left(\frac{1-s}2\right)^{\alpha/2}\right]$ to obtain
\begin{equation}\label{gamma2}
\gamma_\alpha=2\pi \int_{-1}^1\B(s)\left[\left(\frac{1+s}2\right)^{\alpha/2}+
\left(\frac{1-s}2\right)^{\alpha/2}\right]\,ds.
\end{equation}
The integral on the right hand side of \rf{gamma2} is finite because the
function in the brackets is bounded for $s\in [-1,1]$.

In order to show that $\gamma_\alpha>\gamma_2$, whenever $0<\alpha<2$, it 
suffices
to use the elementary inequality
$$
\left(\frac{1+s}2\right)^{\alpha/2}+
\left(\frac{1-s}2\right)^{\alpha/2}>1
$$
which is valid for all $s\in (-1,1)$.
\end{proof}

\begin{corollary} \label{cor:la}
Let $\alpha\in [0,2]$. Assume that the function $(1-s^2)^{\alpha/2} \B(s)$ is 
integrable on
$[-1,1]$. 
For every $\xi \in \RR\setminus \{0\}$ the following quantity
\begin{equation}
\lambda_\alpha \equiv
 \int_{S^{2}} \B\left(\frac{\xi\cdot \sigma}{|\xi|}\right) 
\left( \frac{|\xi^-|^\alpha +|\xi^+|^\alpha}{|\xi|^\alpha }-1\right)\dn
\label{la}
\end{equation}
is finite, independent of $\xi$, and positive provided $0<\alpha<2$.
\end{corollary}

\begin{proof}
If $\B\in L^1(-1,1)$, this is the immediate consequence of Lemma 
\ref{lem:gamma}.
To handle more general $\B$ it suffices to 
apply identities \rf{xipm:alpha} together  
with
the change of variables from \rf{g} with the function
$g=g(s)$ satisfying 
\begin{equation}\label{gB}
0\leq g(s)\equiv \B(s)\left(\left(\frac{1+s}2\right)^{\alpha/2}+
\left(\frac{1-s}2\right)^{\alpha/2}-1\right) \leq C \B(s) (1-s^2)^{\alpha/2}
\end{equation}
for every $\alpha \in (0,2)$, a constant $C(\alpha)>0$, 
and all $s\in [-1,1]$ (see also Remark \ref{rem:comp} below).
\end{proof}

\begin{rmk}\label{rem:comp}
We leave for the reader to check that 
$$
\lim_{s\to \pm 1} \frac{\left(\frac{1+s}2\right)^{\alpha/2}+
\left(\frac{1-s}2\right)^{\alpha/2}-1}{(1-s^2)^{\alpha/2}}=1
$$
provided $\alpha\in (0,2)$. Hence, both functions in the numerator and 
the denominator are comparable in the sense that there are two  
positive constants $C_1,C_2$ such that 
$$
 C_1 (1-s^2)^{\alpha/2}\leq
\left(\frac{1+s}2\right)^{\alpha/2}+
\left(\frac{1-s}2\right)^{\alpha/2}-1 \leq C_2 (1-s^2)^{\alpha/2}.
$$
By this reason, we prefer to use the estimate from \rf{gB} to keep
the assumption on the collision kernel $\B$ from Corollary \ref{cor:la}
comparable with our standing assumption \rf{non-cut}. 
\end{rmk}

\begin{rmk}
In the following, we systematically use the identity 
$\lambda_\alpha=\gamma_\alpha-\gamma_2$ valid for any collision kernel
$\B\in L^1(-1,1)$.
\end{rmk}


\subsection{Construction of solutions}

Now, we are going to construct solutions of problem \rf{Feq}--\rf{Fini}
under the cut-off assumption \rf{cut-off} using the Banach contraction 
principle. The following theorem is a particular case of Theorem 
\ref{thm:exist} under the assumption that $\B$ satisfies \rf{non-cut} with 
$\alpha_0=0$.

\begin{theorem}\label{thm:exist:cut}
Let $\alpha\in [0,2]$ and  $\B\in L^1(-1,1)$.
 For every initial condition  $\vf_0\in \PDa$ there exists 
a unique  classical solution of problem
\rf{Feq}--\rf{Fini} satisfying 
$\vf\in \Xa\equiv C([0,\infty), \PDa)$ 
\end{theorem}

In the proof of Theorem \ref{thm:exist:cut},  we use 
 the following nonlinear operator
\begin{equation}\label{G}
\G(\vf)(\xi)\equiv 
\int_{S^2} \B
\left(\frac{\xi\cdot \sigma}{|\xi|}\right)
\vf(\xi^+)\vf(\xi^-)\dn
\end{equation}
where $\xi^+$ and $\xi^-$ are defined in \rf{xipm}. Hence, under 
the cut-off assumption \rf{cut-off},  for the constant 
$\g=2\pi\int_{-1}^1\B(s)\,ds$ ({\it cf.}~Lemma \ref{lem:gamma}), 
and for $\vf$ satisfying $\vf(0,t)=1$ for all $t\geq 0$, 
we  write equation \rf{Feq} in the following form
\begin{equation}\label{Feq:s}
\partial_t \vf+\g\vf =\G(\vf).
\end{equation}
Next, multiplying \rf{Feq:s} by $e^{\g t}$ and integrating with respect to $t$ 
we obtain the following equivalent formulation of problem
\rf{Feq}--\rf{Fini}
\begin{equation}\label{duhamel}
\vf(\xi,t)=\vf_0(\xi)e^{-\g t}+\int_0^t e^{-\g (t-\tau)} 
\G(\vf(\cdot,\tau))(\xi)\,d\tau.
\end{equation}

\begin{lemma} \label{lem:G:est}
Let $\alpha\in [0,2]$ and assume \rf{cut-off}.
For every $\vf \in\PDa$, the function $\G(\vf)$ is  continuous and 
positive definite.
Moreover, for the constant $\ga$ defined in \rf{gamma}, we have
\begin{equation}\label{G:est}
|\G(\vf)(\xi)-\G(\wvf)(\xi)|\leq \ga \|\vf-\wvf\|_\alpha|\xi|^\alpha
\end{equation}
for all $\vf,\wvf\in \PDa$ and all $\xi\in \RR\setminus \{0\}$.
\end{lemma}

\begin{proof}
Let $\vf\in\PDa$.
To show that $\G(\vf)$ is  continuous and 
positive definite is suffices to follow the reasoning from 
\cite[Lemma 2.1]{PT96}.

 Hence, it suffices to show estimate
\rf{G:est} from all $\vf,\wvf\in\PDa$. To do it,
using   inequalities $|\vf(\xi^-)|\leq 1$, $|\wvf(\xi^+)|\leq 1$,  
and  the definitions of the metric \rf{distance} 
as well as  of the constant $\ga$ (see \rf{gamma}), we obtain
\begin{equation}\label{G:est:1}
\begin{split}
|\G(\vf)(\xi)-\G(\wvf)(\xi)|
=&
\bigg|
\int_{S^2} \B
\left(\frac{\xi\cdot \sigma}{|\xi|}\right) 
\big[
\left(\vf(\xi^+)-\wvf(\xi^+)\right)\vf(\xi^-) \\
&\qquad\qquad\qquad\qquad 
+\wvf(\xi^+)\left(\vf(\xi^-)- \wvf(\xi^-)\right)
\big]\dn
\bigg|\\
\leq& 
\int_{S^2} \B
\left(\frac{\xi\cdot \sigma}{|\xi|}\right) 
\left(
\|\vf-\wvf\|_\alpha |\xi^+|^\alpha+\|\vf-\wvf\|_\alpha |\xi^-|^\alpha
\right)\dn\\
=&\ga\|\vf-\wvf\|_\alpha|\xi|^\alpha
\end{split}
\end{equation}
for all $\xi\in\RR$.
\end{proof}

Now, we are in a position to prove the existence of solutions to 
\rf{Feq}--\rf{Fini} in the space $\PDa$.

\begin{proof}[Proof of Theorem \ref{thm:exist:cut}]
  The solution to \rf{Feq}--\rf{Fini} is obtained as a fixed point of equation 
\rf{duhamel}
via the Banach contraction principle applied to the nonlinear operator
\begin{equation}\label{F}
\F(\vf)(\xi, t)\equiv \vf_0(\xi)e^{-\g t}
+\int_0^t e^{-\g (t-\tau)} \G(\vf(\cdot, \tau))(\xi)\,d\tau
\end{equation}
({\it cf.}~equation \rf{duhamel}).
We fix $\vf_0\in \PDa$ and, first,  
we show that the mapping $\F$  is a contraction 
on the metric space
$
\XaT=C([0,T], \PDa)
$
supplemented with the metric
$
\|\vf-\wvf\|_\XaT \equiv \sup_{\tau\in [0,T]}
\|\vf(\cdot,\tau)-\wvf(\cdot,\tau)\|_\alpha
$
provided $T>0$ is sufficiently small.

Notice that for every $\vf \in \XaT$ and for every $t\in [0,T]$, 
the function $\F(\vf)(t)$ is continuous and positive definite. 
This is the immediate consequence of Lemma \ref{lem:comb} if one approximates
the integral on the right-hand side of \rf{F} by finite sums with positive 
coefficients. Here, one should remember that $\G(\vf)(\tau)$ is 
continuous and 
positive definite for every $\tau\in [0,t]$ by Lemma \ref{lem:G:est}.

Next, for every $\vf \in \XaT$, we can rewrite equation \rf{F} as follows
$$
\F(\vf)(\xi,t)-1=(\vf_0(\xi)-1)e^{-\g t}+\int_0^t e^{-\g (t-\tau)} 
\big(\G(\vf(\cdot,\tau))(\xi)-\g\big)\,d\tau
$$
Hence, using Lemma \ref{lem:G:est} with $\wvf\equiv 1$, 
recalling the definition
of the constant $\g=\G(1)$ from \rf{gamma}, and estimating 
$e^{-\g(t-\tau)}\leq 1$ for every
$\tau \in [0,t]$, we obtain
$$
|\F(\vf)(\xi,t)-1|\leq\|\vf_0-1\|_\alpha |\xi|^\alpha 
+\ga  \int_0^t  \|\vf(\xi,\tau)-1\|_\alpha \,d\tau |\xi|^\alpha.
$$
Dividing this inequality by $|\xi|^\alpha$ and computing the supremum 
with respect to $\xi\in \RR$ and $t\in [0,T]$ we obtain that $\F:\XaT\to\XaT$ 
together with 
the estimate
$$
\|\F(\vf)-1\|_\XaT\leq \|\vf_0-1\|_\alpha + \ga T\|\vf-1\|_\XaT.
$$  

In a similar way, using Lemma \ref{lem:G:est} for every 
$\vf,\wvf \in \XaT$, we get
$$
|\F(\vf)(\xi,t)-\F(\wvf)(\xi,t)|\leq \ga T\|\vf-\wvf\|_\XaT |\xi|^\alpha,
$$
and consequently,
$$
\|\F(\vf)-\F(\wvf)\|_\XaT\leq  \ga T\|\vf-\wvf\|_\XaT.
$$  
Hence, the Banach contraction principle provides the unique solution 
(the fixed point) of equation \rf{duhamel} in the space $\XaT$ provided 
$T<1/\ga$.

Note finally that we have constructed the unique solution on $[0,T]$ where 
$T$ is independent of 
the initial condition. Hence, choosing $\vf(\xi,T)$ as the initial datum
we obtain the
unique solution on $[T,2T]$. Consequently, repeating this procedure, 
we  construct the unique solution
on any finite time interval. 
\end{proof}


\subsection{Remark on Wild's sum}
Under the cut-off assumption \rf{cut-off},
for every characteristic function $\vf_0$ as an initial datum
and for $\g=1$ in \rf{gamma} (which can be always normalized by a suitable time
rescaling of equation \rf{Feq}, {\it cf.} \cite[Sec.~2]{PT96}),
it is possible to derive the following explicit representation
of a classical solution to \rf{Feq}--\rf{Fini}
\begin{equation}\label{w:1}
\vf(\xi,t)=e^{-t}\sum_{n=0}^\infty \vf^{(n)}(\xi)(1-e^{-t})^n,
\end{equation}
where 
\begin{equation}\label{w:2}
\begin{split}
\vf^{(0)}(\xi)&= \vf_0(\xi)\\
\vf^{(n+1)}(\xi)&=\frac{1}{n+1}
\sum_{j=0}^n\widetilde\G\left(\vf^{(j)},\vf^{(n-j)}\right)(\xi)
\end{split}
\end{equation}
with the bilinear operator $\widetilde\G$ 
 of the following form
\begin{equation*}
\widetilde\G(\vf,\wvf)(\xi)\equiv 
\int_{S^2} \B
\left(\frac{\xi\cdot \sigma}{|\xi|}\right)
\vf(\xi^+)\wvf(\xi^-)\dn.
\end{equation*}
Notice that we have   $\widetilde\G(\vf,\vf)=\G(\vf)$ for every $\vf$, 
where $\G$ is defined in  \rf{G}.
This series is called Wild's sum \cite{W51}. 
The proof that it converges toward the
unique classical solution of problem \rf{Feq}--\rf{Fini} can be found {\it e.g.}
either in \cite[Thm 2.1]{PT96} or in \cite[Sect. 4]{BC02-self}.
Here, we  show that the series converges in the space $\PDa$.

\begin{theorem}\label{thm:wild}
Let $\alpha\in (0,2]$ and  $\vf_0\in\PDa$.
Assume that $\g=1$ in \rf{gamma}.
Then the series defined in \rf{w:1}--\rf{w:2} converges toward  a solution
to \rf{Feq}--\rf{Fini} which belongs to the space
$C([0,\infty),\PDa)$. 
\end{theorem}

\begin{proof}
By inspection of the proof of 
Lemma \ref{lem:G:est} with $\wvf\equiv 1$, we immediately obtain the inequality
\begin{equation}\label{WG:est}
\|\widetilde\G(\vf,\wvf)-1\|_\alpha\leq \gamma_\alpha^+\|\vf-1\|_\alpha
+ \gamma_\alpha^-\|\wvf-1\|_\alpha
\end{equation}
for all $\vf,\wvf\in\PDa$, where 
$
\gamma_\alpha^\pm \equiv
\int_{S^2} 
\B\big((\xi\cdot \sigma)/|\xi|\big) |\xi^{\pm}|^\alpha/|\xi|^\alpha \dn
$
are finite, independent of $\xi$, and satisfy $\gamma_\alpha^++\gamma_\alpha^-
=\ga$.

Now, we proceed by induction to show
 the estimate
\begin{equation}\label{vf:n}
\|\vf^{(n)}-1\|_\alpha\leq \gamma_\alpha^n\|\vf_0-1\|_\alpha
\quad \mbox{for every} \quad n\in \mathbb{N},
\end{equation}
where $\vf^{(n)}$ and defined in \rf{w:2} and the constant $\ga\geq \g=1$ 
appers in \rf{gamma}. 
Inequality \rf{vf:n} reduces to an obvious equality if $n=0$. 
For $n\geq 1$, using definition \rf{w:2}, the 
estimate of the bilinear form \rf{WG:est}, 
and the inductive argument, we obtain
\begin{equation}\label{vf:n:2}
\begin{split}
\|\vf^{(n+1)}-1\|_\alpha&\leq 
\frac{1}{n+1}\sum_{j=0}^n \|\widetilde \G\left(\vf^{(j)},\vf^{(n-j)}\right)-
1\|_\alpha\\
&\leq 
\frac{1}{n+1}\sum_{j=0}^n \gamma_\alpha^+ \|\vf^{(j)}-1\|_\alpha
+\gamma_\alpha^-\|\vf^{(n-j)}-1\|_\alpha\\
&=\frac{1}{n+1}\sum_{j=0}^n \ga \|\vf^{(j)}-1\|_\alpha\\
&\leq  \|\vf_0-1\|_\alpha  \left(\frac{1}{n+1}\sum_{j=0}^n \gamma_\alpha^{1+j} 
\right).
\end{split}
\end{equation}
Recall now that $\ga=\gamma_\alpha^++\gamma_\alpha^-\geq \gamma_2\;(=1)$ 
by Lemma \ref{lem:gamma}, hence, 
$\gamma_\alpha^{1+j}\leq \gamma_\alpha^{n+1}$ for each $j\in \{0, ...,n\}$.
Using these inequalities to estimate the right-hand side of 
\rf{vf:n:2} we complete the proof of \rf{vf:n}.

Coming back to the function $\vf(\xi,t)$ given by the series \rf{w:1}
and applying \rf{vf:n}, we obtain
\begin{equation*}
\|\vf(t)-1\|_\alpha\leq e^{-t} \|\vf_0-1\|_\alpha 
\sum_{n=0}^\infty (1-e^{-t})^n \gamma_\alpha^n. 
\end{equation*}
Chosing $T>0$ so small to have $(1-e^{-T})\ga<1$ we obtain the convergence of
the series on the right-hand side for any $t\in [0, T]$.

However, since $T$ from the proof of Theorem \ref{thm:wild} is independent of 
the initial condition, we may  choose $\vf(\xi,T)$ as the initial datum
and show the convergence of the corresponding Wild series on  $[T,2T]$. 
Consequently, repeating this procedure, 
we show the converges of series \rf{w:1}--\rf{w:2}  toward a  solution from 
$C([0,T],\PDa)$ for any $T>0$.
\end{proof}


\subsection{Stability and uniqueness of solutions}

For each $R\in (0, \infty]$, we define the quasi-metric for any 
$\vf,\wvf\in \PDa$ by the following formula
\begin{equation}
\|\vf-\wvf\|_{\alpha,R}\equiv \sup_{|\xi|\leq R}
\frac{|\vf(\xi)-\wvf(\xi)|}{|\xi|^\alpha}.
\end{equation}

The following stability lemma, shown here in the case of the integrable 
collision kernel, will be generalized in Section \ref{sec:non-cut} 
(see Corollary \ref{cor:stab:non-cut})
for
solutions of the initial value problem \rf{Feq}-\rf{Fini} with any nonintegrable
kernel satisfying \rf{non-cut}.

\begin{lemma} \label{lem:stab}
Assume that  $\alpha \in [0,2]$ and $\B\in L^1(-1,1)$.
Consider two solutions
$\vf,\wvf \in C ([0,\infty), \PDa)$ of  problem
\rf{Feq}--\rf{Fini} corresponding to the initial data $\vf_0,\wvf_0\in
\PDa$, respectively.
Then for every $t\geq 0$ and $R\in (0,\infty]$ 
\begin{equation}\label{stab:in}
\|\vf(t)-\wvf(t)\|_{\alpha,R}
\leq e^{\lambda_\alpha t} \|\vf_0-\wvf_0\|_{\alpha,R},
\end{equation}
where the constant $\lambda_\alpha=\ga-\g\geq 0$ is defined in \rf{la0}.
\end{lemma}

\begin{proof}
It follows from equation \rf{Feq} that the function
$$
h(\xi,t)=\frac{\vf(\xi,t)-\wvf(\xi,t)}{|\xi|^\alpha}
$$
satisfies 
\begin{equation}\label{h:eq}
\partial_t h(\xi,t)= \int_{S^{2}} \B\left(\frac{\xi\cdot \sigma}{|\xi|}\right)
\left(
\frac{\vf(\xi^+)\vf(\xi^-)-\wvf(\xi^+)\wvf(\xi^-)}{|\xi|^\alpha} -h(\xi,t)
\right)\dn.
\end{equation}
Now, for $|\xi^+|\leq |\xi|\leq R$ and $|\xi^-|\leq |\xi|\leq R$, 
we use the following inequalities
\begin{eqnarray*}
&&\hspace{-1cm}|\vf(\xi^+)\vf(\xi^-)-\wvf(\xi^+)\wvf(\xi^-)|\\
&&\leq|\vf(\xi^+)-\wvf(\xi^+)||\vf(\xi^-)|+
|\vf(\xi^-)-\psi(\xi^-)||\wvf(\xi^+)|\\
&& \leq \|\vf(t)-\wvf(t)\|_{\alpha,R}(|\xi^+|^\alpha +|\xi^-|^\alpha)
\end{eqnarray*}
to deduce from equation \rf{h:eq} that
the function $h=h(\xi,t)$ satisfies
\begin{equation}
\left|\partial_th(\xi,t)+\g h(\xi,t)\right|\leq \ga \|\vf(t)-
\wvf(t)\|_{\alpha,R} 
\label{h1}
\end{equation}
with the constants $\ga$ and $\g$ defined in \rf{gamma}.
It follows from inequality \rf{h1} 
$$
\left|\partial_t \left(e^{t\g}h(\xi,t)\right)\right|\leq 
\ga e^{t\g} \|\vf(t)-\wvf(t)\|_{\alpha,R}
$$
for every $t>0$,
hence,
\begin{equation*}
\begin{split}
\left|e^{t\g}h(\xi,t)\right|&
\leq |h(\xi,0)|+ \int_0^t |\partial_s   \left(e^{s\g}h(\xi,s)\right) |\,ds \\
&\leq |h(\xi,0)|+ \ga \int_0^t e^{s\g}
 \|\vf(s)-\wvf(s)\|_{\alpha,R}\,ds.
\end{split}
\end{equation*}
Finally, we compute the supremum with respect to $|\xi|\leq R$ and we apply the
Gronwall lemma to obtain
$$
 \|\vf(t)-\wvf(t)\|_{\alpha,R}
\leq  \|\vf_0-\wvf_0\|_{\alpha,R} e^{t(\ga-\g)}.
$$
The proof is complete because $\ga-\g=\lambda_\alpha$ by \rf{la}.
\end{proof}


\section{Nonintegrable collision kernels -- 
existence, uni\-que\-ness, and stability of solutions}\label{sec:non-cut}

In this section, we complete the proofs of Theorems \ref{thm:exist}
and \ref{thm:stab}
on  the existence and the stability of 
solutions to \rf{Feq}-\rf{Fini} with no cut-off assumption imposed on the 
collision kernel. More precisely, we assume that $\B$ satisfies 
\rf{non-cut} for some $\alpha_0\in [0,2]$.

As a standard practice, 
we consider the increasing sequence of bounded collision kernels
\begin{equation}\label{Bn}
\B_n(s)\equiv\min \{\B(s), n\}\leq \B(s), \quad n\in \N,
\end{equation}
and, for each $\alpha\in [\alpha_0,2]$, 
 the  sequence 
$\vf_n\in C([0,\infty), \PDa)$ 
 of the corresponding
solutions to problem \rf{Feq}--\rf{Fini} (constructed in Theorem 
\ref{thm:exist:cut})
with the kernels $\B_n$ 
and with the same initial datum $\vf_0\in\PDa$.
Note that, under the assumption \rf{non-cut} ({\it cf.} Corollary 
\ref{cor:la}), we have
\begin{equation}\label{la:n}
\lambda_{\alpha,n}\equiv
 \int_{S^{2}} \B_n\left(\frac{\xi\cdot \sigma}{|\xi|}\right) 
\left( \frac{|\xi^-|^\alpha +|\xi^+|^\alpha}{|\xi|^\alpha }-1\right)\dn
\leq \lambda_\alpha
\end{equation}
for every $n\in \N$,
hence by the stability lemma \ref{lem:stab} with $R=\infty$, it follow
\begin{equation}\label{stab:n}
\|\vf_n(t)-1\|_\alpha\leq e^{\lambda_{\alpha,n} t}\|\vf_0-1\|_\alpha
\leq e^{\lambda_\alpha t}\|\vf_0-1\|_\alpha
\end{equation}
for all $t\geq 0$.

\begin{lemma}\label{lem:AA}
Assume that the collision kernel satisfies \rf{non-cut}
with some $\alpha_0\in [0,2]$.
Let 
$\alpha\in [\alpha_0,2]$. 
The sequence of  solutions $\{\vf_n\}_{n=1}^\infty\subset 
C([0,\infty),\PDa)$
is bounded in  $C(\RR\times [0,\infty))$ and equicontinuous.
\end{lemma}

\begin{proof} {\it Step 1. Uniform bound.}
Since $\vf_n(\cdot,t)$ is  a characteristic function
 for every $t\geq 0$, by \rf{vf:2},
we have $|\vf_n(\xi,t)|\leq \vf_n(0,t)=1$ for all $\xi \in \RR$ and $t\geq 0$.

{\it Step 2. Modulus of continuity in time.} 
We use the equation satisfied by $\vf_n$ and inequalities \rf{vf:alpha} 
and \rf{stab:n} 
as follows (remember that $\vf_n(0,t)=1$)
\begin{equation*}
\begin{split}
|\partial_t \vf_n(\xi,t)|&\leq 
\int_{S^2} \B_n
\left(\frac{\xi\cdot \sigma}{|\xi|}\right)
\big|\vf_n(\xi^+,t)\vf_n(\xi^-,t)-\vf_n(\xi,t)\vf_n(0,t)\big|\dn\\
&\leq 4 \|\vf_n(t)-1\|_\alpha \int_{S^2} \B_n
\left(\frac{\xi\cdot \sigma}{|\xi|}\right) |\xi^+|^{\alpha/2}|\xi^-|^{\alpha/2}
\dn\\
&\leq 4\beta_\alpha e^{\lambda_\alpha t} \|\vf_0-1\|_\alpha|\xi|^\alpha
\end{split}
\end{equation*}
for all $\xi\in\RR$ and $t\geq 0$.
Here, 
$\beta_\alpha$ denotes the  finite and independent of $\xi$ number which,
by 
 identities \rf{xipm:alpha} and by the change of variables
from \rf{g}, satisfies
 \begin{equation*}
\begin{split}
\beta_\alpha&\equiv
\int_{S^2} \B
\left(\frac{\xi\cdot \sigma}{|\xi|}\right) 
\frac{|\xi^+|^{\alpha/2}|\xi^-|^{\alpha/2}}{|\xi|^\alpha}\dn\\
&=
\int_{S^2} \B
\left(\frac{\xi\cdot \sigma}{|\xi|}\right)
\left(\frac{1+\frac{\xi}{|\xi|}\cdot \sigma}{2}\right)^{\alpha/4}
\left(\frac{1-\frac{\xi}{|\xi|}\cdot \sigma}{2}\right)^{\alpha/4}\dn\\
&=2\pi \int_{-1}^1 \B(s) \left(\frac{1+s}{2}\right)^{\alpha/4}
\left(\frac{1-s}{2}\right)^{\alpha/4}\,ds.
\end{split}
\end{equation*}
Since $\alpha\in [\alpha_0,2]$, it is now clear that the constant $\beta_\alpha$
is finite for any collision kernel $\B$ satisfying \rf{non-cut}.

{\it Step 3. Modulus of continuity in space.}
It suffices to apply inequality  \rf{ineq:1} 
combined with \rf{ReIm} and \rf{stab:n}  
to obtain the estimate
\begin{equation*}
\begin{split}
|\vf_n(\xi,t)-\vf_n(\eta,t)|&\leq
\sqrt{2(1-\Re \vf_n(\xi-\eta,t))}\\
&\leq \sqrt2 |\xi-\eta|^{\alpha/2}\|\vf_n(t)-1\|_\alpha^{1/2}\\
&\leq \sqrt2 |\xi-\eta|^{\alpha/2}e^{\lambda_\alpha t/2 } \|\vf_0-
1\|_\alpha^{1/2},
\end{split}
\end{equation*}
for all $t\geq 0$, 
where the right-hand side is independent of $n$.
\end{proof}

Now, we are in a position to construct a solution to \rf{Feq}-\rf{Fini}.
By Lemma \ref{lem:AA}, the Ascoli-Arzel\`a theorem, and the Cantor
diagonal argument,  we deduce that 
there exists a subsequence of solutions $\{\vf_{n_k}\}_{n_k}$ 
converging uniformly in every compact set of $\RR\times [0,\infty)$.
We are going prove that the function 
\begin{equation}\label{lim}
\vf(\xi,t)=\lim_{n_k\to \infty}\vf_{n_k}(\xi,t)
\end{equation}
is a solution of
problem \rf{Feq}--\rf{Fini} with the singular kernel $\B$
satisfying \rf{non-cut}. Note here that $\vf(\cdot,t)$ is a characteristic
function for every $t\geq 0$ as the pointwise limit of characteristic 
functions.

Here, we are allowed to use the Lebesgue dominated convergence 
 theorem to pass to 
the limit $n_k\to\infty$ in the Boltzmann operator
\begin{equation}\label{Bnk}
\int_{S^2} \B_{n_k}
\left(\frac{\xi\cdot \sigma}{|\xi|}\right)
\big(\vf_{n_k}(\xi^+,t)\vf_{n_k}(\xi^-,t)
-\vf_{n_k}(\xi,t)\vf_{n_k}(0,t)\big)\dn
\end{equation}
(where, as usual $ \vf_{n_k}(0,t)=1$) because, following the calculations from Step 2 of the proof 
of Lemma \ref{lem:AA}, its integrand can be majorized by the integrable 
(on $S^2$)
function 
$$
4 e^{\lambda_\alpha t} \|\vf_0-1\|_\alpha
\B
\left(\frac{\xi\cdot \sigma}{|\xi|}\right) 
|\xi^+|^{\alpha/2}|\xi^-|^{\alpha/2}.
$$
Since, the Boltzmann operator form \rf{Bnk} converges uniformly on every
compact subset of $\RR\times [0,\infty)$, there exists a continuous 
function $\zeta=\zeta(\xi,t)$ such that $\partial_t\vf_{n_k}\to \zeta$.
By the limit relation \rf{lim}, we immediately conclude that $\zeta=\partial_t 
\vf$.
Hence, $\vf$ is a solution to the initial-value problem \rf{Feq}-\rf{Fini}.

To show the $\vf(\cdot,t)\in \PDa$, it suffices to pass to the pointwise 
limit 
$n_k\to \infty$ in inequality \rf{stab:n} written in the following equivalent
way
$$
\frac{|\vf_{n_k}(\xi,t)-1|}{|\xi|^\alpha}\leq 
e^{\lambda_\alpha t}\|\vf_0-1\|_\alpha
$$
for all  $\xi\in\RR\setminus \{0\}$ and $t\geq 0$.

In order to prove stability inequality form Theorem \ref{thm:stab}, it suffices
to consider two sequences of solutions $\{\vf_{n}\}_{n\in\N}$ and 
$\{\wvf_{n}\}_{n\in\N}$
to equation \rf{Feq} with the truncated 
kernel $\B_n$ and 
corresponding to the initial conditions $\vf_0$ and $\wvf_0$, respectively.
By the compactness argument form Lemma \ref{lem:AA}, there exists a subsequence
$n_k\to \infty$ and solutions $\vf$, $\wvf$ to equation \rf{Feq}   such that 
$$ 
\vf(\xi,t)=\lim_{n_k\to \infty}\vf_{n_k}(\xi,t) \quad\mbox{and}\quad
\wvf(\xi,t)=\lim_{n_k\to \infty}\wvf_{n_k}(\xi,t).
$$
Using the stability lemma \ref{lem:stab} and estimate \rf{la:n}, we obtain
$$
\frac{|\vf_{n_k}(\xi,t)-\wvf_{n_k}(\xi,t)|}{|\xi|^\alpha}\leq 
e^{\lambda_\alpha t}\|\vf_0-\wvf_0\|_\alpha
$$
for all  $\xi\in\RR\setminus \{0\}$ and $t\geq 0$.
Passing to the limit $n_k\to \infty$, we complete the proof of stability 
inequality
\rf{stab:in:0} which, in particular, implies the uniqueness of solutions
to \rf{Feq}-\rf{Fini} in the space $C([0,\infty),\PDa)$.

An analogous argument allows us to remove the cut-off assumption from
the stability lemma \ref{lem:stab}.

\begin{corollary}\label{cor:stab:non-cut}
Assume that $\B$ satisfies the non cut-off  condition \rf{non-cut}
for some $\alpha_0\in [0,2]$.
For every $\alpha\in [\alpha_0,2]$ and $R>0$, the
stability estimates \rf{stab:in} from Lemma \ref{lem:stab} 
hold true
for solutions 
to problem \rf{Feq}-\rf{Fini} with the kernel $\B$.
\end{corollary}

\section{Self-similar solutions  by Bobylev and Cercignani}

In this section, we are going to formulate
(in a way the most suitable for our applications)
 results by Bobylev and Cercignani
\cite{BC02-self} on solutions $(\mu, \VF)$ to equation
\begin{equation}\label{Feq-self2}
\mu \eta \cdot \nabla \VF(\eta) =
\int_{S^2} \B\left(\frac{\eta\cdot \sigma}{|\eta|}\right)
\big(\VF(\eta^+)\VF(\eta^-)-\VF(\eta)\VF(0)\big)\dn.
\end{equation}
Recall that, in this case,  the function $\vf(\xi,t)=\VF(\xi e^{\mu t})$ is 
the self-similar solution of equation \rf{Feq}.

Let us first compute  the scaling 
parameter $\mu$ in equations \rf{Feq-self2} for any collision kernel $\B$ 
satisfying the weaker assumption 
\begin{equation}\label{non-cut0}
(1-s)^{\alpha/2}(1+s)^{\alpha/2} \B(s)\in L^1(-1,1)
\quad \mbox{for some} \quad \alpha \in [0,2].
\end{equation}

\begin{lemma}\label{lem:lambda-alpha}
Assume that the collision kernel satisfies the assumption \rf{non-cut0} for 
some $\alpha\in [0,2]$.
Let $\VF$ be a $C^1$-solution of \rf{Feq-self2} with the following
properties
\begin{equation}\label{self:as}
\VF(\eta)=\VF(|\eta|) \quad \mbox{and}\quad 
\lim_{|\eta|\to 0} \frac{\VF(|\eta|)-1}{|\eta|^\alpha}=K
\end{equation}
 for some $K\neq 0$. Then
\begin{equation}\label{mu:self}
\mu= \frac1\alpha 
 \int_{S^{2}} \B\left(\frac{\eta\cdot \sigma}{|\eta|}\right) 
\left( \frac{|\eta^-|^\alpha +|\eta^+|^\alpha}{|\eta|^\alpha }-1\right)\dn
\end{equation}
\end{lemma}

\begin{proof}
Since, $\VF$ is radially symmetric, by the Hospital rule, we obtain
$$
\lim_{|\eta|\to 0} \frac{\eta\cdot \nabla \VF(\eta)}{|\eta|^\alpha}
=
\lim_{|\eta|\to 0} \frac{\VF'(|\eta|)}{|\eta|^{\alpha-1}}
=\alpha \lim_{|\eta|\to 0} \frac{\VF(|\eta|)-1}{|\eta|^\alpha}=\alpha K.
$$
On the other hand, since $|\eta^\pm|\leq |\eta|$ and $\VF(0)=1$, by assumption \rf{self:as}, 
we get
\begin{equation}\label{mu:self2}
\begin{split}
&\frac{\VF(\eta^+)\VF(\eta^-)-\VF(\eta)\VF(0)}{|\eta^-|^\alpha +|\eta^+|^\alpha-
|\eta|^\alpha}\\
&\qquad\qquad=
\frac{\frac{\VF(\eta^+)-1}{|\eta^+|^\alpha}\VF(\eta^-)|\eta^+|^\alpha
+\frac{\VF(\eta^-)-1}{|\xi^-|^\alpha}|\eta^-|^\alpha
-\frac{\VF(\eta)-1}{|\eta|^\alpha}|\eta|^\alpha}{|\eta^-|^\alpha 
+|\eta^+|^\alpha-|\eta|^\alpha}\to K
\end{split}
\end{equation}
as $|\eta|\to 0$. 
Hence, dividing equation \rf{Feq-self2} by $|\eta|^\alpha$, 
passing to the limit $|\eta|\to 0$ by the Lebesgue dominated convergence
theorem in the integral on the right-hand side, 
and using relation \rf{mu:self2}, 
we  obtain equality~\rf{mu:self}. 
\end{proof}

According to Lemma \ref{lem:lambda-alpha} and Corollary \ref{cor:la}, 
for each $\alpha \in [0,2]$ we introduce the constant
\begin{equation}\label{mu:a}
\mu_\alpha = \frac{\lambda_\alpha}{\alpha} =\frac1\alpha
 \int_{S^{2}} \B\left(\frac{\eta\cdot \sigma}{|\eta|}\right) 
\left( \frac{|\eta^-|^\alpha +|\eta^+|^\alpha}{|\eta|^\alpha }-1\right)\dn
\end{equation}
which is finite and independent of $\eta$ 
if $\B$ satisfies assumption \rf{non-cut0}.

\begin{theorem}[Bobylev \& Cercignani \cite{BC02-self}]
\label{thm:BC:self}
Assume that the collision kernel $\B$ satisfies 
the weaker non cut-off assumption \rf{non-cut0} for some $\alpha \in (0,2)$.
For every 
$K< 0$  and for $\mu=\mu_\alpha$ defined in \rf{mu:a} there exists a radially symmetric solution
$\VF=\VF_{\alpha,K}\in\PDa$ of equation \rf{Feq-self2} 
satisfying
\begin{equation}\label{PHI:0:K}
 \lim_{|\eta|\to 0} \frac{\VF_{\alpha,K}(\eta)-1}{|\eta|^\alpha}=K.
\end{equation}
\end{theorem}

\begin{proof}[Sketch of proof.]
This result was shown in
\cite[Thm.~6.2]{BC02-self}. 
Let us  sketch that proof for the reader convenience 
 and for the completeness of our exposition.

The authors of \cite{BC02-self} look for radially symmetric solutions 
of equation \rf{Feq-self2}. Hence, introducing the function 
\begin{equation}\label{vf:radial}
\phi(x) =\VF(\eta) \quad\mbox{where} \quad x=\frac{|\eta|^2}{2}
\end{equation}
and using identities \rf{xipm2} combined with the change 
of variables from \rf{g}, they reduce equation \rf{Feq-self2} 
to
\begin{equation}\label{Feq:rad}
2\mu \partial_x \phi(x)=\int_0^1 G(s)\big( \phi(sx)\phi((1-s)x)-\phi(0)
\phi(x)\big)\,ds,
\end{equation}
where 
\begin{equation}\label{GB}
G(s)=4\pi \B(1-2s)\quad\mbox{for}\quad s\in (0,1).
\end{equation}
Now, to keep our notation consistent with that 
used in \cite{BC02-self}, we have to introduce the parameter
\begin{equation}\label{aa:w}
\wa=\frac{\alpha}{2}\in (0,1].
\end{equation}
A solution to equation \rf{Feq:rad} is obtained in the form o the series
\begin{equation}\label{phi:series}
\phi(x)=\phi_{\wa,K}(x)=\sum_{n=0}^\infty 
\frac{u_n x^{n\wa}}{\Gamma (n\wa+1)}
\end{equation}
with the coefficients defined by the recurrence formula 
\begin{equation}\label{un}
u_0=1, \quad
u_1=\sqrt2 K,\quad
u_n=\frac{1}{\gamma(\wa,n)}\sum_{j=1}^{n-1}
B_\wa(j,n-j)u_ju_{n-j}\;\; \mbox{for}\; n\ge2. 
\end{equation}
Here,
\begin{align}
\gamma(\wa,n)&=n \lambda(\wa)-\lambda(n\wa)\label{g1}\\
\lambda(p)&=\int_0^1 G(s) \big(s^p+(1-s)^p-1\big)\,ds\\
B_\wa(j,\ell)&= 
\frac{\Gamma(n\wa+1)}{\Gamma(j\wa+1)\Gamma(\ell\wa+1)}
\int_0^1 G(s)s^{j\wa}(1-s)^{\ell\wa}\,ds.\label{Bjl}
\end{align}
Next, the reasoning from \cite{BC02-self} consists in showing that the 
series \rf{phi:series}-\rf{un} converges toward a solution of \rf{Feq:rad}.
The proof that $\phi(|x|^2/2)$ is a characteristic function is written 
in \cite[p.~1054]{BC02-self}.

Coming back to our original notation, we obtain a solution to \rf{Feq-self2} in 
the form of the series
\begin{equation}\label{VF:series}
\VF_{\alpha,K}(\eta)= \sum_{n=0}^\infty 
\frac{u_n 2^{-n/2} (|\eta|^\alpha)^{n}}{\Gamma (n\alpha/2+1)},
\end{equation}
where $u_n$ is defined in \rf{un}.
Obviously, this limit function belongs to $\PDa$ and satisfies relation 
\rf{PHI:0:K} by the definition of first two elements of 
the sequence $u_n$ from \rf{un}. 
\end{proof}

\begin{rmk}\label{rmk:K<0}
Bobylev and Cercignani, in their proof of Theorem \ref{thm:BC:self}, show the
convergence of the series \rf{VF:series} without any sign condition imposed
on the constant $K$ (in fact, complex $K$ is also allowed). Let us explain 
that if we limit ourselves to characteristic functions satisfying 
\rf{PHI:0:K}, then necessarily 
$K\leq 0$. Indeed, first notice that 
$K$ has to be a real number because,
by using the identity $\VF(-\eta)=\overline{\VF(\eta)}$
we have  
\begin{equation*}
K=\lim_{|\eta|\to 0} \frac{\VF(\eta)-1}{|\eta|^\alpha}
=\lim_{|\eta|\to 0} \frac{\VF(-\eta)-1}{|\eta|^\alpha}
=\lim_{|\eta|\to 0}\overline{ \frac{\VF(\eta)-1}{|\eta|^\alpha}}
=\overline{K}.
\end{equation*}
This means that, in particular, we have
\begin{equation}\label{K:Re:lim}
K=\lim_{|\eta|\to 0} \frac{\Re \VF(\eta)-1}{|\eta|^\alpha}.
\end{equation}
The right-hand side of inequality \rf{K:Re:lim} is nonpositive because of
Lemma \ref{lem:Re} and inequality~\rf{vf:2}.
\end{rmk}

\begin{rmk}
Theorem \ref{thm:BC:self} is shown under the assumption that the function from 
\rf{GB} satisfies the estimate
$0\leq G(s)\leq Cs^{-(1+\gamma)}$
for some constants $\gamma\in (0,1)$,  $C>0$ and for all $s\in (0,1)$, see 
\cite[Assump.~(A) on p.~1052]{BC02-self}.
It is clear, however, from the proof by Bobylev and Cercignani 
that their reasoning holds true provided the quantities in \rf{g1}--\rf{Bjl}
are finite. Hence, one can assume, for example, that 
$(s(1-s))^\gamma G(s) \in L^1(0,1)$. Using the formula \rf{GB} we discover
our assumption \rf{non-cut0} with $\gamma=\alpha/2$. 
\end{rmk}

\begin{rmk}\label{rem:opt}
Now, it is clear that the estimate of 
the growth of the quantity $\|\vf(t)- \wvf(t)\|_\alpha$
expressed by
inequality \rf{stab:in:0} is optimal. Indeed, this can be
easily seen when we substitute in \rf{stab:in:0}  
the self-similar solution $\vf(\xi,t)=\VF_{\alpha,K}(\xi e^{\mu_\alpha})$
and $\wvf\equiv 1$.
In this special case, since $\mu_\alpha=\lambda_\alpha/\alpha$,
we have
\begin{eqnarray*}
\|\vf(t)-1\|_\alpha=
e^{\lambda_\alpha t}\sup_{\xi\in\RR}
\frac{\left| \VF_{\alpha,K}(\xi e^{\mu_\alpha t}) -1\right|}{|\xi e^{\mu_\alpha 
t}|^\alpha} 
= e^{\lambda_\alpha t} \|\VF_{\alpha,K}-1\|_\alpha
\end{eqnarray*}
for all $t>0$.
\end{rmk}


\section{Asymptotic stability of solutions}\label{sec:asymp}

Now, we are ready to prove our main result on the large time asymptotics 
of solutions to \rf{Feq}-\rf{Fini} with a nonintegrable collision kernel
satisfying \rf{non-cut}.

\begin{proof}[Proof of Theorem \ref{thm:asymp}]
First, we apply  the stability inequality \rf{stab:in}
with some $R\in (0,\infty)$ which is now valid, by Corollary 
\ref{cor:stab:non-cut}, in the case of any kernel satisfying \rf{non-cut}.

Using relation \rf{phi=psi}
we substitute   
$\vf(\xi,t)=\psi(\xi e^{\mu_\alpha t},t)$ 
and $\wvf(\xi,t)= \wpsi(\xi e^{\mu_\alpha t},t)$ into 
inequality \rf{stab:in} to obtain
\begin{equation}\label{esas:1}
\sup_{|\xi|\leq R}\frac{|\psi(\xi e^{\mu_\alpha t},t)
-\wpsi(\xi e^{\mu_\alpha t},t)|}{|\xi|^\alpha}
\leq e^{\lambda_\alpha t}\sup_{|\xi|\leq R}\frac{|\psi_0(\xi)
-\wpsi_0(\xi)|}{|\xi|^\alpha}
\end{equation}
for all $t>0$ and each $R\in (0,\infty]$. Next, it follows from equality
$\alpha\mu_\alpha=\lambda_\alpha$ that 
\begin{equation*}
\sup_{|\xi|\leq R}\frac{|\psi(\xi e^{\mu_\alpha t},t)
-\wpsi(\xi e^{\mu_\alpha t},t)|}{|\xi|^\alpha}
= e^{\lambda_\alpha t}
\sup_{|\xi|\leq R\, e^{\mu_\alpha t}}\frac{|\psi(\xi,t)
-\wpsi(\xi,t)|}{|\xi|^\alpha},
\end{equation*}
hence, by \rf{esas:1},
\begin{equation}\label{esas:2}
\sup_{|\xi|\leq R\, e^{\mu_\alpha t}}\frac{|\psi(\xi,t)
-\wpsi(\xi,t )|}{|\xi|^\alpha}
\leq \sup_{|\xi|\leq R}\frac{|\psi_0(\xi)
-\wpsi_0(\xi)|}{|\xi|^\alpha}
\end{equation}
for all $t>0$ and each $R\in (0,\infty]$. Since $R$ is arbitrary, we are allowed
to substitute $R=S e^{-\mu_\alpha t}$ in \rf{esas:2} (when $S$ will be 
chosen later on) to obtain
 \begin{equation}\label{esas:3}
\sup_{|\xi|\leq S}\frac{|\psi(\xi,t)-\wpsi(\xi,t )|}{|\xi|^\alpha}
\leq \sup_{|\xi|\leq S \,e^{-\mu_\alpha t}}\frac{|\psi_0(\xi)
-\wpsi_0(\xi )|}{|\xi|^\alpha}
\end{equation}

Now, we are in a position to complete the proof. Recall that 
$|\psi(\xi,t)|\leq 1$ and $|\wpsi(\xi,t)|\leq 1$. 
Hence, for every $\varepsilon>0$ 
there exists $S>0$ such that 
\begin{equation}\label{esas:4}
\sup_{|\xi|> S}\frac{|\psi(\xi,t)-\wpsi(\xi,t )|}{|\xi|^\alpha} \leq 
\frac{2}{R^\alpha}\leq \varepsilon.
\end{equation}
Consequently, with this choice of $S$,  by \rf{esas:3} and \rf{esas:4}, 
we have
\begin{equation}\label{esas:5}
\begin{split}
\|\psi(t)-\wpsi(t)\|_\alpha &\leq
\sup_{|\xi|\leq S}\frac{|\psi(\xi,t)-\wpsi(\xi,t )|}{|\xi|^\alpha}+
\sup_{|\xi|> S}\frac{|\psi(\xi,t)-\wpsi(\xi,t )|}{|\xi|^\alpha}\\
&\leq  \sup_{|\xi|\leq S \,e^{-\mu_\alpha t}}\frac{|\psi_0(\xi)
-\wpsi_0(\xi )|}{|\xi|^\alpha} +\varepsilon.
\end{split}
\end{equation}
By the assumption on $\psi_0$ and $\wpsi_0$ (see \rf{psi:0}),
we 
immediately obtain that the first term on the right hand side of \rf{esas:5}
tends to zero as $t\to\infty$. Since, $\varepsilon>0$ can be arbitrary small
we complete the proof of Theorem \ref{thm:asymp}.
\end{proof}

\begin{rmk}
Corollary
\ref{cor:asymp:self} implies that solutions 
the original problem \rf{eq}-\rf{ini},
which converge 
toward self-similar profile by Bobylev and Cercignani
(in the sense stated in Corollary \ref{cor:asymp:self}), cannot have 
finite energy. 
Indeed, by the Toscani and Villani result recalled in Remark \ref{rem:Max}, finite energy solutions
to \rf{eq}--\rf{ini} have to converge in the metric  $\|\cdot\|_2$  toward Maxwellian.  
This fact is in contrast with a result by Mischler and Wennberg \cite{MW99} 
who showed that any solution
of the homogenous Boltzmann equation which satisfies certain bounds on moments
of order $\alpha<2$ must necessarily have also have bounded energy.
However, they consider the equation with so-called hard potential.
As we have explained, 
such a phenomenon cannot be true for Maxwellian molecules.
\end{rmk}



\end{document}